\documentclass{amsart}
\usepackage{final} % Macros are defined here
\usepackage{xcolor}
\usepackage{enumitem}
\usepackage{scalerel,stackengine}
\usepackage{geometry}
\usepackage[
backend=biber, 
hyperref=true, 
style=ext-numeric, 
citexref=true,
url=false,
doi=false,
isbn=false,
eprint=true,
articlein=false,
giveninits=true,
date=year,
maxnames=5,
minnames=5,
maxalphanames=5,
minalphanames=5
]{biblatex}  
\usepackage{hyperref}
\hypersetup{
	colorlinks=true,
	linkcolor=blue,
	citecolor = blue,      
	urlcolor=black,
}

\DeclareFieldFormat[online]{date}{Last Modified: #1}  % Add "last modified" before date in online items
\DeclareFieldFormat[article]{pages}{#1}  % Do not include pp for articles
\DeclareFieldFormat
[article,inbook,incollection,inproceedings,patent,thesis,unpublished]
{title}{#1}  % Remove quotes from article titles
\DeclareFieldFormat
[article,inbook,incollection,inproceedings,patent,thesis,unpublished]
{titlecase:title}{\MakeSentenceCase*{#1}} %%% Make article titles in sentence case
\DeclareFieldFormat[article,periodical]{number}{\mkbibparens{#1}}
\DeclareFieldFormat{issuedate}{#1}
\usepackage{amssymb}
\allowdisplaybreaks
\numberwithin{equation}{section}

\newtheorem{theorem}{Theorem}[section]
\newtheorem{lemma}[theorem]{Lemma}
\newtheorem{proposition}[theorem]{Proposition}
\newtheorem{corollary}[theorem]{Corollary}
\newtheorem{fact}[theorem]{Fact}

\theoremstyle{definition}

\newtheorem{definition}[theorem]{Definition}
\newtheorem{example}[theorem]{Example}
\newtheorem{remark}[theorem]{Remark}

\title[Diffraction of plane waves, spherical waves, and beyond]{Diffraction of plane waves, spherical waves, \\ and beyond}

\author{Emily R. Korfanty}
\address{Department of Mathematical and Statistical Sciences, University of Alberta, \newline \indent Edmonton, AB,  T6G 2G1, Canada \\
}
\email{ekorfant@ualberta.ca}
\urladdr{}

\author{Jan Maz\'{a}\v{c}}
\address{Fakult\"at f\"ur Mathematik, Universit\"at Bielefeld,
	\newline \indent Postfach 100131, 33501 Bielefeld, Germany}
\email{jmazac@math.uni-bielefeld.de}
\urladdr{}

\date{\today}

\begin{document}

\begin{abstract}
We review the diffraction theory for plane waves and establish its connection to the diffraction of Besicovitch almost periodic functions, extending the theory to an unbounded setting and providing explicit formulas.
Then, we give an alternative proof that the diffraction of a spherical wave in $\RR^d$ is a~single sphere, which was recently shown in \cite{BKM25}.
After developing a suitable framework for working with spherically symmetric measures, including a radial analogue of the usual Lebesgue decomposition,
we introduce the notion of radial almost periodicity.  In particular, we define a space of Besicovitch radially almost periodic functions and show that this space contains precisely the functions whose radial part is Besicovitch almost periodic. The paper concludes with a~diffraction analysis of these functions. 
\end{abstract}

\maketitle

\section{Introduction}
Diffraction analysis is a well-established tool in materials science, physics, and also in harmonic analysis. It is essential to the theory of aperiodic order, where it offers a better understanding of long-range order for structures such as tilings or Delone sets. Moreover, the diffraction measure provides essential dynamical information, most notably in the pure-point case, where the diffraction and dynamical spectra coincide \cite{BL2004}. 
A long-standing unsolved problem asks for the exact diffraction of a tiling with statistical radial symmetry, known as the pinwheel tiling; see \cite{Rad1994}. It has been shown that this diffraction pattern is radially symmetric \cite{MPS2006}, and numerical computations suggest that it consists of concentric circles \cite{BFG2007, BFG2007a, GD2011}, but the exact nature of this diffraction measure remains an open question.
Since this remains a very challenging problem, one can take partial steps by studying objects with radially symmetric diffraction patterns.

In the recent work \cite{BKM25}, we constructively answered the following question posed by Strungaru in the affirmative:

\vspace{0.7mm}

\noindent
\textit{Is there a planar structure (say, a bounded function, or a
      translation-bounded measure) whose diffraction pattern is uniformly
      distributed and concentrated on a single circle?}

\vspace{0.7mm}
    
In particular, we established that the spherical wave \( f(x) = \ee^{2\pi \ii r \|x\|} \) has this property in any dimension.  In this paper, we build on this result by investigating the diffraction of series of plane waves and spherical waves, respectively.   

We consider the standard infinite-volume limit approach to diffraction. First, recall that the natural autocorrelation $\eta$ for a complex-valued, locally integrable and bounded function $f$ in $d$ dimensions is defined by
\begin{equation}\label{eq:auto-def}
   \eta(x) \, = \lim_{R\to\infty} 
  \myfrac{f\big|^{}_R \ast \widetilde{f\big|^{}_R}(x)}{\vol (B^{}_R )} \, = \, \lim_{R\to\infty} \myfrac{\vG \bigl(\frac{d}{2} +
    1 \bigr)}{\pi^{\frac{d}{2}} \, R^d} \int^{}_{B^{}_R} f (y) \,
  \overline{f (x-y)}\ts  \dd y \ts , \end{equation}
provided the limit exists (if this is not the case, one can still proceed along a~subsequence; see \cite[Sec.~9.1]{BG2013} for details). We write $B^{}_R$ for the ball of radius $R$ centred at the origin, $f\big|^{}_{R}$ for the restriction of $f$ to $B^{}_{R}$, and for any function $g$, we have $\widetilde{g} (x) \defeq \overline{g(-x)}$. We write $\vG$ for the gamma function, with $\vG(x+1) = x \vG(x)$,
and with values $\vG\bigl(\tfrac{1}{2}\bigr)=\sqrt{\pi}$ and
$\vG(1)=\vG(2)=1$.   The second equality in Eq.~\ref{eq:auto-def} is due to \cite[Lem.~1.2]{Sch2000}, which effectively allows us to replace $\widetilde{f\big|_{R}}$ with $\widetilde{f}$; this is often more convenient for
computations. In certain cases, this lemma does not apply and one must proceed directly; see Proposition \ref{thm:Bap}.

By construction, $\eta$ is a positive definite function, and is thus Fourier transformable (in the sense of tempered
distributions). Its Fourier transform is a positive measure by the
Bochner--Schwartz theorem, which is known as the \emph{diffraction
  measure} of $f$; we refer to \cite[Ch.~9]{BG2013} and references
therein for background. 
For our present purposes, it
suffices to think in terms of tempered distributions, which avoids some
subtleties of the Fourier analysis of unbounded measures \cite{AG1974}.

In this paper, we expand on recent studies of Besicovitch almost periodic functions \cite{LSS2020,LSS2024,LSS2024b} by establishing the existence of the natural autocorrelation measure for series of plane waves and spherical waves (and their limits in the corresponding Besicovitch spaces), respectively. Furthermore, we provide explicit computations for the natural autocorrelation and diffraction measures of these waves. We do so without a boundedness assumption on the underlying function.  To our knowledge, this is the first class of functions for which the existence of the autocorrelation measure can be established in such generality, although several results concerning unbounded settings are known; see \cite{BSV2020,SV2016}. Furthermore, we apply an abstract result recently developed in \cite{LS2024} to derive a criterion (Corollary~\ref{cor:orthogonal}) for the presence of circles in diffraction, which is based on the structure itself (and not on the autocorrelation).

While the plane wave case is somewhat trivial, the corresponding theory gives us a tool for studying the diffraction properties of Besicovitch almost periodic functions explicitly, and emphasises why the Besicovitch space is a particularly natural choice for studying diffraction. 
When increasing the dimension, plane waves lead very naturally
to the diffraction structure of lattices and model sets (or
mathematical quasicrystals), which are ultimately consequences of
Poisson's summation formula and its various generalisations, see \cite{LLRSS2020,RS2017}.  
Although we have a radial analogue of Poisson's summation formula \cite{BFG2007}, and a~natural connection to spherically symmetric measures has been established \cite{AKM2025}, much less is known about
the diffraction of spherically symmetric structures in general, motivating our study of spherical waves and their diffraction. The spherical wave case requires more subtle techniques, and to the authors' best knowledge, it has yet to be rigorously derived; we aim to fill this gap.

In our analysis of spherical waves, we introduce the space of \emph{Besicovitch radially almost periodic functions}, denoted by $\Brap^{\ts 2}_{}(\RR^d)$, which consists of functions given by series of the form
\[
    F(x) \, = \, \sum_{m=1}^\infty c^{}_m \ee^{2\pi \ii a^{}_m \|x\|} \,,
\]
where convergence is taken with respect to the Besicovitch $2$-seminorm (see Definition~\ref{def:Besicovitch-seminorm}).   
Our main result demonstrates the suitability of this space: the autocorrelation $\eta_F$ of any such function always exists, and the diffraction consists of spheres of radius $|a^{}_m|$ and amplitudes $|c^{}_m|^2$, as one may expect; see Theorem~\ref{thm:Brap-diffraction}. 

Although we study spherical waves and their diffraction, our setting differs from that of  \cite{BHPI,BHPII,BHPIII}, which generalise the diffraction measure to arbitrary locally compact second countable groups via the spherical Fourier transform; in contrast, we are focused on extending the Euclidean theory. A more closely related work is \cite{Ami1997}, which examines the Fourier spectrum for radially periodic functions within a different mathematical framework.

The paper is structured as follows. In Section~\ref{sec:preliminaries_and_background}, we introduce spheres in $\RR^d$ as measures and recall their connection to the Bessel functions of the first kind via the Fourier transform. We also recall the reflected Eberlein convolution, its properties, and its role in the theory of diffraction.  Then, in Section~\ref{sec:besicovitch_almost_periodic_functions}, we introduce the Besicovitch almost periodic functions and use their properties to analyse the convergence of their autocorrelation measures.  In particular,  Corollary~\ref{cor:approx_Bap} describes why the Besicovitch space is the suitable one for studying autocorrelation, and Theorem~\ref{thm:Bap} extends the definition of the autocorrelation measure by proving its existence for \emph{all} Besicovitch almost periodic functions, without the usual restriction to $L^\infty(\RR^d)$.  
In Section~\ref{sec:plane}, we review plane waves and their diffraction, and prove an explicit formula for the diffraction of Besicovitch almost periodic functions (Theorem~\ref{thm:Besicovitch-diffraction}).  In Section~\ref{sec:spherical}, we consider the diffraction of spherical waves and their superpositions. We first introduce the notions of radially concentrated versus radially dispersed measures, giving a radial counterpart to the usual Lebesgue decomposition for measures. We discuss the aforementioned criterion for the existence of circles in diffraction (Corollary~\ref{cor:orthogonal}), and then introduce the space of Besicovitch radially almost periodic functions. Finally, we prove radial versions of the results obtained for plane waves (Proposition~\ref{prop:spherical_auto_ex}), and prove our main results on the diffraction of spherical waves (Theorem~\ref{thm:Brap-diffraction}).

\section{Preliminaries and background}\label{sec:preliminaries_and_background}
Throughout this paper, we use $\lambda$ to denote the Lebesgue measure on $\RR^d$.  We denote the volume of a~set $S \subseteq \RR^d$ by $\vol(S) \deq \lambda(S)$, and write $\boldsymbol{1}^{}_{S}$ for the characteristic function of $S$.  

We write $L^\infty(\RR^d)$ for the space of essentially bounded functions on $\RR^d$, and $\Cu(\RR^d)$ for the space of uniformly continuous functions.  We employ the  notation $\|f\|^{}_\infty \defeq \sup_{x\in\RR^d}|f(x)|$ in various contexts, where applicable.

As is standard in diffraction theory, we will always use the term \emph{measure} to refer to any complex Radon measure, i.e., any continuous linear functional on the space $\Cc(\RR^d)$ of continuous compactly supported functions, equipped with an appropriate inductive limit topology; see \cite[Sec.~8.5.1]{BG2013} for details.  

By the Fourier transform of an integrable function $h$ on $\RR^d$, we mean 
\[
    \hh (x) \, = \int_{\RR^d} \ee^{-2 \pi \ii x \cdot y} h(y) \ts \dd y \ts ,
\]
which naturally extends to finite, and then to translation-bounded measures.  A measure $\mu$ is said to be translation bounded if the condition
\[
    \sup_{t \in \RR^d} |\mu|(t+K) \, < \, \infty \ts ,
\]  
holds for all compact sets $K\subseteq \RR^d$, where $|\mu|$ denotes the variation of $\mu$.  

\subsection{Spheres and Bessel functions}
We will use the following notation for spheres and their counterparts as probability measures.

\begin{definition}
\label{def:spher_meas}
Let $\theta^{}_{r}$ with $r>0$ denote the probability measure for the
uniform distribution on $\Sph^{d-1}_r$, the sphere of radius $r>0$
in $\RR^d$ centred at $0$, where we set $\Sph_r^0 = \{ -r, r\}$.    
\end{definition}

The Fourier
transform of $\theta^{}_{r}$, compare \cite[Eq.~9]{BFG2007}, is the
continuous function
\begin{equation}\label{eq:sphere-Bessel}
    \widehat{\ts\ts\theta^{}_{r}} (k) \, =\,  \int_{\RR^d} \ee^{-2 \pi \ii k \cdot x}\ts
    \dd \theta^{}_{r} (x) \, = \, \vG\! \lt( \myfrac{d}{2}\rt)
    \frac{J^{}_{\frac{d}{2}-1}\bigl( 2 \pi r \| k \| \bigr)}
    {\bigl( \pi r \| k \| \bigr)^{\frac{d}{2} - 1}} \ts .
\end{equation}
Recall that $J_\nu$ is an entire function with series expansion
\begin{equation}\label{eq:Bessel-series}
   J_{\nu} (z) \, = \, \lt( \myfrac{z}{2}\rt)^{\nu} \sum_{m=0}^{\infty}
   \myfrac{(-1)^{m}}{m! \, \vG(\nu+m+1)} \lt(\myfrac{z}{2}\rt)^{2m} \ts .
\end{equation}
Additionally, $J_\nu$ admits the following integral representation 
\begin{equation}\label{eq:Bessel-integral}
J_\nu(z) \, = \, \myfrac{\lt(\frac{z}{2}\rt)^{\nu}}{\sqrt{\pi}\vG \! \lt(\nu+\frac{1}{2}\rt)} \int_0^\pi \ee^{\pm \ii z\cos\theta} \sin^{2\nu}\theta \,\dd\theta \ts ,
\end{equation}
whenever $\mathrm{Re}(\nu)>-\frac{1}{2}$.  This representation will be helpful for our computations in Section~\ref{sec:spherical}. For more details and various other representations of $J^{}_{\nu}$, see \cite[Ch.~9-10]{AS1972}.

\subsection{Reflected Eberlein convolution}
We now recall the general definition of the~reflected Eberlein convolution, which was recently introduced in \cite{LSS2024b}.

In what follows, the term \emph{van Hove sequence} refers to a sequence $\cA = (A^{}_n)^{}_{n\in \mathbb{N}}$ of precompact Borel subsets of $\RR^d$ such that, for any compact set $K\subseteq \RR^d$,  the limit 
\[
    \lim_{n\to \infty} \myfrac{\partial^K(A^{}_n)}{\vol(A^{}_n)} \, = \, 0 \ts ,
\]
where
\[
    \partial^K(A^{}_n) \, = \, \vol((K+A^{}_n)\backslash A_n^\circ) \cup ((-K + \overline{\RR^d \backslash A^{}_n})\cap A^{}_n)) \ts,
\]
denotes the $K$-boundary of $A^{}_n$, and $A_n^\circ$ denotes the interior of $A^{}_n$.  

Note that for all compact sets $K$ and all $u \in K$, both $(u+A^{}_n) \backslash A^{}_n$ and $A^{}_n \backslash (u+A^{}_n)$ are contained in $\partial^K(A^{}_n)$.
In particular, every van Hove sequence is also \emph{F{\o}lner}, i.e., 
\[
    \lim_{n\to \infty} \myfrac{A^{}_n \Delta (u + A^{}_n)}{\vol(A^{}_n)} \, = \, 0 \ts,
\]
holds for every $u \in \RR^d$. For more details, see \cite{PRS22}.

\begin{definition}[Reflected Eberlein convolution]\label{def:reflected}
    Let $\cA=(A^{}_n)^{}_{n\in\NN}$  be a van Hove sequence in $\RR^d$.
    \begin{enumerate}[label=(\alph*)]
     \item Let  $f,g\in L^\infty(\RR^d)$. The \emph{reflected Eberlein convolution} of $f$ and $g$ is the function~$\lb f, g \rb^{}_{\cA}$  defined by the limit
    \[
    \lb f, g \rb^{}_{\cA}(x) \, \defeq \, \lim_{n\to \infty} \myfrac{1}{\vol(A^{}_n)} \int^{}_{A^{}_n} f(s)\ts \overline{g(s-x)} \,\dd s  \, = \, \lim_{n\to \infty} \myfrac{\bigl(\res{A^{}_{n}}{f}\ast \,\widetilde{g}\bigr)(x)}{\vol(A^{}_n)} \ts ,
    \]
    provided that this limit exists for all $x\in \RR^d$. In this case, we say that $\lb f, g \rb^{}_{\cA}$ exists.  
    We write $\lb f,g \rb$ instead of $\lb f, g \rb^{}_{\cA}$ when the choice of the van Hove sequence is clear from context.
    \item We say that measures $\mu, \nu$ have a well-defined reflected Eberlein convolution along $\cA$ if
    \[\lb \mu, \nu \rb^{}_{\cA} \, = \, \lim_{n\to \infty} \myfrac{1}{\vol(A^{}_{n})} \res{A^{}_{n}}{\mu} \ast  \,\widetilde{\res{A^{}_{n}}{\nu}} \ts , \]
    exists in the vague topology. Here, $\res{A^{}_{n}}{\mu}$ represents the restriction of $\mu$ to $A^{}_{n}$, and for a measure $\mu$, we define $\widetilde{\mu}$ by $\widetilde{\mu}(g) \defeq \overline{\mu(\widetilde{g})}$. 
    We write $\lb \mu, \nu \rb$ instead of $\lb \mu, \nu \rb^{}_{\cA}$ when the choice of the van Hove sequence is clear from context.  
    \end{enumerate} 
\end{definition}

The definitions of the reflected Eberlein convolution for functions and for measures differ slightly. Whereas for functions, we effectively restrict only the function $f$, in the definition for measures, we restrict both measures simultaneously. This has the advantage that the resulting object is a positive definite measure by construction. Nevertheless, it turns out that this small difference does not play a role as long as we stay within the $L^{\infty}$ class for functions. The proof of the following is an easy consequence of the F{\o}lner property satisfied by all van Hove sequences, and can be found in \cite[Lem.~1.2]{Sch2000}.

\begin{proposition}\label{prop:bounded-conv}
Let $f,g \in L^\infty(\RR^d)$ and let $\cA = (A^{}_n)^{}_{n\in\NN}$ be a van Hove sequence.  Then the following equality holds for all $x \in \RR^d$:
\begin{equation}\label{eq:equality}
    \lim_{n\to \infty} \myfrac{1}{\vol(A^{}_n)} \int_{A^{}_n} f(s) \overline{g(s-x)} \, \dd s \, = \, \lim_{n\to \infty} \myfrac{1}{\vol(A^{}_n)} \int_{\RR^d} \res{A^{}_n}{f}(s) \overline{\res{A^{}_n}{g}(s-x)} \, \dd s \ts .
\end{equation}  
\end{proposition}

If the assumption $f,g\in L^\infty(\RR^d)$ is dropped, the equality in Eq.~\ref{eq:equality} could potentially fail.  
Therefore, to avoid confusion, we only use the notation $\lb f, g \rb_\cA$ when $f,g \in L^\infty(\RR^d)$. 

It is easy to see that Definition~\ref{def:reflected} (a) yields the natural autocorrelation $\eta(x)$ for $f$ given in Eq.~\eqref{eq:auto-def}, provided that it exists. In the following definition, we clarify our notation for the autocorrelation of functions and extend the definition to measures in a natural way.

\begin{definition}[Autocorrelation]\label{def:autocorrelation}
    Let $\cA=(A^{}_n)^{}_{n\in\NN}$  be a van Hove sequence in $\RR^d$.
    \begin{enumerate}[label=(\alph*)]
    \item Let $f \in L^\infty(\RR^d)$.  The \emph{autocorrelation of $f$ along $\cA$} is given by
    \[
        \eta^{}_{f}(x)\, \defeq \, \lb f , f \rb^{}_{\cA}(x) \,,
    \]
    provided that this limit exists for all $x\in \RR^d$. 
    \item Let $\mu$ be a measure.  We define the \emph{autocorrelation of $\mu$ along $\cA$} to be the measure
    \[
        \gamma^{}_\mu \, \defeq \,  \lb \mu, \mu \rb^{}_{\cA} \,,
    \]
    provided that the limit exists in the vague topology.  
    \end{enumerate} 
\end{definition}

\begin{remark}
Given a translation-bounded measure~$\mu$ and a~van Hove sequence $\cA$, the autocorrelation $\gamma^{}_\mu$ always exists along a subsequence of $\cA$, as demonstrated by Hof \cite{Hof1995}. Moreover, the reflected Eberlein convolution is closely related to the usual Eberlein convolution 
\[
    f \circledast^{}_{\cA} g (x) \, = \, \lim_{n\to \infty} \myfrac{1}{\vol(A^{}_n)} \int^{}_{A^{}_n} f(s)\ts g(x-s) \,\dd s \ts .
\]
With this definition, we have $\lb f, g \rb^{}_{
\cA} = f \circledast^{}_{\cA} \widetilde{g}$. In particular, $\lb f, g \rb^{}_{\cA}$ exists if and only if $f \circledast^{}_{\cA} \widetilde{g}$ exists \cite[Rem.~3.3]{LSS2024b}.  We prefer to work with the reflected Eberlein convolution due to several structural advantages from the perspective of diffraction; see \cite{LSS2024b} for details. 

All of the above concepts can also be defined in the abstract setting of locally compact Abelian groups, where one replaces the van Hove sequence with a more general van Hove net. In our Euclidean setting, using sequences would not bring any advantage. However, we will often use the collection of balls $(B^{}_{R})^{}_{R>0}$,  
which is not explicitly a sequence (as it is not indexed by a countable set of indices), but a~trivial net with ordering inherited from the reals. 
\exend
\end{remark}

The following proposition reconciles the definitions of autocorrelation for functions and measures in the case when the measure is translation bounded.

\begin{proposition}[{\cite[Prop.~1.4]{LSS2020}}]\label{prop:autocorrelation-reflected}
Let $\mu$ be a translation-bounded measure.
If $\mu$ is absolutely continuous with respect to $\lambda$ with a uniformly continuous density function $f$, then the autocorrelation $\gamma^{}_\mu$ is also absolutely continuous with density function $\eta^{}_{f}$, i.e., 
\[
    \gamma^{}_\mu = \eta^{}_{f}\lambda \,.
\] 
\end{proposition}

\begin{remark}\label{rem:non-TB-challenges}
If the assumption that $\mu = \lambda f$ is translation bounded is removed from Proposition~\ref{prop:autocorrelation-reflected}, the following challenges could potentially arise:
$\gamma^{}_{\mu}$ may not exist, $\eta^{}_{f}$ may not exist, or $\gamma^{}_{\mu}$ may exist but differ from $\eta^{}_{f}\lambda$.
Another subtle observation is the following. If $f \in L^\infty(\RR^d)$, then it is easy to show that $\mu = f\lambda$ is translation bounded; however, the converse does not hold. Given $f \in L^1(\RR^d)$, the estimate 
\[
|\mu|(x+K) \, \leqslant \|f\|^{}_1\ts ,
\] 
shows that $\mu = f\lambda$ is always a translation-bounded measure, even when $f$ is unbounded. \exend
\end{remark}

In the proofs to come, we will often use the following (sesqui-)linearity property of the reflected Eberlein convolution.

\begin{proposition}[{\cite[Prop.~3.11]{LSS2024b}}]\label{prop:linear}
    The reflected Eberlein convolution is linear in the first argument and sesquilinear in the second argument, i.e., for any $a,b\in\CC$ and $f,g,h\in L^{\infty}(\RR^d)$, we have
    \[
    \lb af+bh,g \rb^{}_{\cA} \, =\, a\ts\ts \lb f, g \rb^{}_{\cA}  + b\ts \ts \lb h, g \rb^{}_{\cA} \ts ,
    \]
    and
    \[
    \lb f, ag+bh\rb^{}_{\cA}  \, = \, \overline{a}\ts\ts \lb f, g \rb^{}_{\cA}  + \overline{b}\ts\ts \lb f, h \rb^{}_{\cA} \ts ,
    \]
    provided that these reflected Eberlein convolutions exist.
\end{proposition}

We conclude this section with a theorem that provides a way to study certain diffraction properties of a~measure by working directly with the underlying structure instead of the autocorrelation measure. This is especially relevant for the many applications where the diffraction remains unknown due to challenges in computing the autocorrelation.

\begin{theorem}[{\cite[Thm.~6.1]{LS2024}}]\label{thm:orthogonal}
Let $\mu$, $\nu$ be translation-bounded measures and let $\cA = (A^{}_n)^{}_{n\in\NN}$ be a van Hove sequence.  If the autocorrelations $\gamma^{}_\mu$ and $\gamma^{}_\nu$ exist along $\cA$ and $\widehat{\gamma^{}_\mu} \perp \widehat{\gamma^{}_\nu}$, then $\lb\mu, \nu\rb^{}_{\cA}$ exists and $\lb\mu,\nu\rb^{}_{\cA} = 0$. 
\end{theorem}

\section{Besicovitch almost periodic functions}\label{sec:besicovitch_almost_periodic_functions}
The theory of almost periodic functions goes back to Harald Bohr \cite{Bohr}, who studied the (nowadays called) \emph{Bohr (strong) almost periodic} functions. Recall that a bounded and uniformly continuous $\CC$-valued function on the real line $\RR^d$ is Bohr almost periodic if, for all $\varepsilon>0$, the set 
$\{ t \in \RR^d \ : \ \|f - T^{}_{t}f \|^{}_{\infty} <\varepsilon \}$ of $\varepsilon$-almost periods is relatively dense in $\RR^d$. Here, we define $(T^{}_t f)(x) \defeq f(x-t)$. The Bohr almost periodic functions can also be understood as the closure of the set of trigonometric polynomials with respect to the supremum norm \cite[Thm.~4.3.5]{MS2017}.  The set of all Bohr almost periodic functions on $\RR^d$ is denoted $\SAP(\RR^d)$. Any $f \in \SAP(\RR^d)$ has pure point diffraction; this result goes back to Bohr \cite{Bohr}, and can also be found, for example, in \cite[Thm.~2.2.3]{Benedetto}. 

In this section, we recall a particular class of almost periodic functions with respect to the following seminorm introduced by Besicovitch \cite{Bes1926}.   
\begin{definition}
Let $1\leqslant  p < \infty$ and let~$\cA=(A^{}_n)^{}_{n\in\NN}$ be a van Hove sequence in $\RR^d$.
Given a 
    function $f\in \Lloc{p}(\ts \RR^d)$, we define 
    \[
    \bigl\|f\bigr\|^{}_{b,p,\cA} \, \defeq \, \limsup_{n\to \infty} \lt(\myfrac{1}{\vol(A^{}_n)} \int_{A^{}_n}|f(x)|^p \, \dd x \rt)^{\frac{1}{p}} \ts .
    \]
    We call this  object the \emph{Besicovitch $p$-seminorm}; indeed, it acts as a seminorm on the space
    \[
    BL^{p}_{\cA}(\RR^d) \, \defeq \, \bigl\{f\in \Lloc{p}(\RR^d) \,:\, \bigl\|f\bigr\|^{}_{b,p,\cA} < \infty\bigr\}\ts .
    \]
    When the choice of the van Hove sequence is clear from context, we will write~$\|\cdot\|_{b,p}$ instead of~$\|\cdot\|_{b,p,\cA}$.
\end{definition}
We can define an equivalence relation $\equiv$  on the space $BL^{p}_{\cA}(\RR^d)$ as 
\begin{equation}
\label{eq:equiv}
    f \equiv g \,  \Longleftrightarrow \, \|f - g\|^{}_{b,p,\cA} = 0 \, .
\end{equation}
Later, we will use this equivalence to construct Hilbert spaces from various subspaces of $BL^{p}_{\cA}(\RR^d)$. 

\begin{remark}
\label{rem:net_vs_sequence}
   We will often work with $\cA = (B^{}_{R})^{}_{R>0}$, which does not pose any difficulties, as  
    \[ 
    \limsup_{R} \lt(\myfrac{1}{\vol(B^{}_{R})} \int_{B^{}_{R}}|f(x)|^p \, \dd x \rt)^{\frac{1}{p}}  \, = \, \limsup_{n\to \infty} \lt(\myfrac{1}{\vol(B^{}_{n})} \int_{B^{}_{n}}|f(x)|^p \, \dd x \rt)^{\frac{1}{p}} \, 
    \]
    whenever one of the sides exists. \exend
\end{remark}

Now, if we consider all functions that can be approximated by trigonometric polynomials with respect to the Besicovitch seminorm, we obtain the space of Besicovitch $p$-almost periodic functions.

\begin{definition}\label{def:Besicovitch-seminorm}
Let $1\leqslant  p < \infty$ and let~$\cA=(A^{}_n)^{}_{n\in\NN}$ be a van Hove sequence in~$\RR^d$. A~locally integrable function $f\in \Lloc{p}(\RR^d)$ is called a \emph{Besicovitch $p$-almost periodic function} if, for each $\eps>0$, there exists a~trigonometric polynomial $P$ such that 
\[
\bigl\|f - P\bigr\|^{}_{b,p,\cA} \, < \, \eps \ts .
\]
The space of all Besicovitch $p$-almost periodic functions is denoted by $\Bap\ts^{p}_{\cA}(\RR^d)$.
\end{definition}

\begin{remark}
For the function spaces mentioned above, we have the following inclusions \cite{LSS2020}.
\[ 
\begin{array}{rccc}
    & \Bap\ts^{p}_{\cA}(\RR^d) & \subseteq & BL^{p}_{\cA}(\RR^d)\\
 &\cup \ \raisebox{0.5pt}{$\scriptstyle{\mathrm{dense}}$}  & & \cup  \\
\SAP(\RR^d)\  \overset{\mathrm{dense}}{\subset}  & \Bap\ts^{p}_{\cA}(\RR^d) \cap L^{\infty}(\RR^d) & \subseteq & L^{\infty}(\RR^d)
\end{array}\, , 
\]
where it follows from the definition that $L^\infty(\RR^d)\subseteq BL^{p}_{\cA}(\RR^d)$ for any $1 \leqslant  p < \infty$ and for any van Hove sequence $\cA$. 
    \exend
\end{remark} 

It was shown in \cite{LSS2024b} that any $f \in L^\infty(\RR^d) \cap \Bap^2_{\cA}(\RR^d)$ has pure point diffraction.

\begin{proposition}[{\cite[Thm.~3.30]{LSS2024b}}]\label{prop:reflected-exists}
Let $\cA$ be a van Hove sequence. The reflected Eberlein convolution $\lb f, f \rb^{}_{\cA}$ exists and is strongly almost periodic for any function $f \in L^\infty(\RR^d) \cap \Bap_{\cA}^2(\RR^d)$.  In particular, $f$~has pure point diffraction. 
\end{proposition}

We will extend this result to the space $\Bap^2_{\cA}(\RR^d)$ by showing that the diffraction of any $2$-Besicovitch almost periodic function is pure point, with no boundedness assumption; see Theorem~\ref{thm:Besicovitch-diffraction}. 

A key feature of the space $\Bap^2_{\cA}(\RR^d)$ is the Hilbert space structure described in the following theorem, which allows us to describe any Besicovitch almost periodic function as a trigonometric series.  

\begin{theorem}[{\cite[Thm.~3.19]{LSS2020}}]\label{thm:hilbert}
Let $\cA$ be a van Hove sequence and let $\equiv$ denote the equivalence relation~\eqref{eq:equiv} on $BL^2_{\cA}(\RR^d)$. Then, the space $\Bap^2_{\cA}(\RR^d) / \equiv$ is a Hilbert space with inner product defined by
\[
    \langle [f], [g] \rangle^{}_{\cA} \, = \, \lim_{n\to \infty} \myfrac{1}{\vol(A^{}_{n})} \int_{A^{}_{n}} f(t) \overline{g(t)} \,\dd t \,,
\]
for which $\{[\ts\ee^{2\pi \ii x \cdot a}\ts] \,:\, a \in \RR^d\}$ is an orthonormal basis. 
In particular, we have $f \in \Bap^2_{\cA}(\RR^d)$ if and only if there exist sequences $(c^{}_{m})^{}_{m \in \NN}$ in $\CC$ and $(a^{}_{m})_{m\in\NN}$ in~$\RR^d$ such that 
\[
    f(x) = \sum_{m=1}^\infty c^{}_m \ee^{2\pi \ii x \cdot a^{}_m} \,,
\]
for all $x \in \RR^d$, where the infinite sum converges with respect to $\|\cdot\|^{}_{b,2,\cA}$.
\end{theorem}
\begin{remark}
    We emphasise that the elements of $\Bap^2_{\cA}(\RR^d) / \equiv$  are equivalence classes, which contain functions that can be quite different. In particular, every continuous function with compact support belongs to the equivalence class of the zero function. Therefore, the notion of convergence with respect to the Besicovitch seminorm is fairly weak and does not convey any information about pointwise or uniform convergence. \exend
\end{remark}

\begin{remark}
    The Hilbert space structure requires a van Hove sequence. For a general van Hove net, one can only establish the structure of a~pre-Hilbert space, see \cite[Sec.~3]{LSS2020}. On the other hand, the Hilbert space structure is still guaranteed if one works with sufficiently nice van Hove sequences, which is true for our setting due to Remark~\ref{rem:net_vs_sequence}. \exend
\end{remark}

\pagebreak
Another key property of $\Bap^2_{\cA}(\RR^d)$ is the following Cauchy--Schwarz type inequality for the reflected Eberlein convolution. 

\begin{proposition}[{\cite[Prop.~3.11]{LSS2024b}}]\label{prop:CS}
    Let $\cA$ be a van Hove sequence.  For any $f,g \in L^\infty(\RR^d)\cap BL_{\cA}^{2}(\RR^d)$, we have
    \begin{equation}
    \label{eq:CS}
    \bigl\|\lb f, g \rb^{}_{\cA} \bigr\|^{}_\infty \, \leqslant  \,  \bigl\|f\bigr\|^{}_{b,2,\cA}\ts \bigl\|g\bigr\|^{}_{b,2,\cA} \ts ,
    \end{equation}
    provided that $\lb f, g \rb^{}_{\cA}$ exists.
\end{proposition}

We immediately obtain the following continuity result showing that the Besicovitch seminorm is a very natural object to work with when dealing with autocorrelation functions, as the autocorrelation remains unchanged when modified by a function of zero Besicovitch seminorm. 
\begin{corollary}
\label{cor:approx_Bap}
Let $\cA$ be a van Hove sequence. For any $f,g \in L^\infty(\RR^d) \cap \Bap_{\cA}^2(\RR^d)$, we have
\[
    \bigl\|f-g \bigr\|^{}_{b,2,\cA}\, = \, 0 \quad \Longrightarrow \quad  \eta^{}_{f} \,=\, \eta^{}_{g} \ts.
\]
\end{corollary}
\begin{proof}
By Proposition~\ref{prop:reflected-exists}, we have that $\lb f, f \rb^{}_{\cA}$ and $\lb g, g \rb^{}_{\cA}$ exist. Then, we have
\begin{align*}
 \bigl\|  \lb f, f \rb^{}_{\cA}  - \lb g, g \rb^{}_{\cA} \bigr\|^{}_\infty &\leqslant \, \bigl\|  \lb f, f \rb^{}_{\cA}  - \lb f, g \rb^{}_{\cA} \bigr\|^{}_\infty   + \bigl\|  \lb f, g \rb^{}_{\cA}  - \lb g, g \rb^{}_{\cA} \bigr\|^{}_\infty \\[2pt]
 & \leqslant \, \bigl(\| f \|^{}_{b,2,\cA} + \| g \|^{}_{b,2,\cA}) \ts \bigl\| f-g \bigr\|^{}_{b,2,\cA} =0 \,. \qedhere
\end{align*}
\end{proof}

\begin{corollary}\label{cor:Besicovitch-ac}
    Let $\cA$ be a van Hove sequence and let $f\in L^\infty(\RR^d)\cap \Bap_{\cA}^{2}(\RR^d)$.  If $(P^{}_N)^{}_{N\in\NN}$ is a~sequence of trigonometric polynomials with $\bigl\|f-P^{}_N\bigr\|^{}_{b,2} \longrightarrow 0$, then 
    \[
    \bigl\|\eta^{}_f - \eta^{}_{P^{}_N}\bigr\|^{}_\infty \, \longrightarrow \, 0 \ts .
    \]
\end{corollary}

\begin{proof}
By Proposition~\ref{prop:reflected-exists}, we have that $\lb f, f \rb^{}_{\cA}$, $\lb f, P^{}_N \rb$, and $\lb P^{}_N, P^{}_N \rb $ exist.  Then, we have
\[
    \bigl\|\eta^{}_f - \eta^{}_{P^{}_n}\bigr\|^{}_\infty 
    \, \leqslant  \, \bigl\|f-P^{}_n\bigr\|^{}_{b,2} \ts\bigl(\bigl\|f-P^{}_n\bigr\|^{}_{b,2} + 2\bigl\|f\bigr\|^{}_{b,2}\bigr)\ts .
\]
Since $\bigl\|f\bigr\|^{}_{b,2} < \infty$ and $\bigl\|f - P^{}_n\bigr\|^{}_{b,2}\longrightarrow 0$, the claim follows.
\end{proof}

Although the space $L^\infty(\RR^d)\cap \Bap^{\, 2}(\RR^d)$ is dense in $\Bap^{\, 2}(\RR^d)$, the preceding approach cannot extend the result to all $f\in \Bap^{\, 2}(\RR^d)$ because it relies on results which only hold on $L^\infty (\RR^d)$. Nevertheless, considering denseness, one may still expect the result to extend to the entire space.

To conclude this section, we show that the same result still holds for all $f\in \Bap^{\, 2}(\RR^d)$ --- the autocorrelation of a function $f$ is a limit of the autocorrelations of the approximating polynomials.  This is the content of Theorem~\ref{thm:Bap}, our main result for this section. The proof avoids using the reflected Eberlein convolution, which is not well defined on the entire space $\Bap^{\, 2}(\RR^d)$; see Remark \ref{rem:non-TB-challenges}. Therefore, we start with two lemmas concerning convergence properties of objects that we will use in the proof of Theorem~\ref{thm:Bap}. These lemmas can also be seen as a refinement of the result by Schlottmann \cite[Lem.~1.2]{Sch2000}.

\begin{lemma}
\label{lem:uniform-convergence-compact}
Let $\cA = (A^{}_n)^{}_{n\in\NN}$ be a van Hove sequence. If $f \in \Lloc{1}(\RR^d)$ is such that 
\[
    \myfrac{1}{\vol(A^{}_n)} \res{A^{}_{n}}{f} \ast \,\widetilde{\res{A^{}_{n}}{f}} \longrightarrow h
\] 
uniformly on compact sets, then the autocorrelation $\gamma^{}_\mu$ for $\mu = f\lambda$ exists along $\cA$, and we have
\[
    \gamma^{}_\mu \, = \, h\lambda \ts .
\]
\end{lemma}
\begin{proof}
Let $\varphi \in \Cc(\RR^d)$ and let $K$ be a compact set containing $\supp(\varphi)$.  For each~$n$, we have 
\[
    \res{A^{}_{n}}{(f\lambda)} \ast \widetilde{\,\res{A^{}_{n}}{(f\lambda)}} \,(\varphi) \, = \, \int_K \varphi(s) \bigl(\res{A^{}_{n}}{f}* \widetilde{\,\res{A^{}_{n}}{f}}\bigr)(s) \ts  \dd s \,,
\]
which implies
\begin{align*}
    \lim_{n\to \infty} & \left|\myfrac{1}{\vol(A^{}_n)} \res{A^{}_{n}}{(f\lambda)} \ast \widetilde{\,\res{A^{}_{n}}{(f\lambda)}}(\varphi) \,- \, \int_{\RR^d} \varphi(s) h(s) \ts \dd s\right| \\
    & \qquad \qquad \leqslant \, \lim_{n\to \infty} \|\varphi\|^{}_\infty \int_K \left| \myfrac{1}{\vol(A^{}_n)} \bigl( \res{A^{}_{n}}{f}* \widetilde{\,\res{A^{}_{n}}{f}}\bigr)(s) - h(s) \right| \ts \dd s \, =  \, 0\ts .
\end{align*} 
This shows that for all $\varphi \in \Cc(\RR^d)$, we get 
\[
    \lim_{n\to \infty} \myfrac{1}{\vol(A^{}_n)} \res{A^{}_{n}}{(f\lambda)} \ast \widetilde{\, \res{A^{}_{n}}{(f\lambda)}} \, (\varphi) \, = \, \int_{\RR^d} \varphi(s) h(s) \,\dd s \ts .
\]
Therefore, the sequence $\frac{1}{\vol(A^{}_n)} \res{A^{}_{n}}{(f\lambda)} \ast \widetilde{\, \res{A^{}_{n}}{(f\lambda)}}$ converges to $h\lambda$ vaguely, proving the claim.
\end{proof}

\begin{lemma}\label{lem:P^{}_N-uniform}
Let $\cA = (A^{}_n)^{}_{n\in\NN}$ be a van Hove sequence,  let $P(x) = \sum_{k=1}^N c^{}_k \ee^{2 \pi \ii x \cdot y^{}_k}$ be a trigonometric polynomial, and let $h(x) = \sum_{k=1}^N |c^{}_k|^2\ee^{2 \pi \ii x \cdot y^{}_k}$.  Then the following statements hold.
\begin{enumerate}[label=(\alph*)]
\item The sequence $\frac{1}{\vol(A^{}_n)} \res{A^{}_{n}}{P} \ast \, \widetilde{P}$ converges uniformly to $h$;
\item The sequence $\frac{1}{\vol(A^{}_n)} P \ast \, \widetilde{\res{A^{}_{n}}{P}}$ converges uniformly to $h$;
\item The sequence $\frac{1}{\vol(A^{}_n)} \res{A^{}_{n}}{P} \ast \, \widetilde{\res{A^{}_{n}}{P}}$ converges uniformly on compact sets to $h$\,.
\end{enumerate}
In particular, the autocorrelation $\gamma^{}_\mu$ for $\mu = P\lambda$ exists along $\cA$ and 
\[
\gamma^{}_\mu \, = \, h\lambda \,.
\]
\end{lemma}

\begin{proof}
A straightforward computation shows
\[
\bigl(\res{A^{}_{n}}{P} \ast \, \widetilde{P}\ts \bigr)(u) \,= \, \sum_{k, \ell=1}^N  c^{}_k \ts \overline{c^{}_\ell} \ts \ee^{2\pi \ii u\cdot y^{}_\ell}\int_{A^{}_n} \ee^{2\pi \ii s\cdot (y^{}_k- y^{}_\ell)} \, \dd s \ts ,
\]
which leads to the  estimate
\begin{equation}\label{eq:P^{}_N-uniform-est-1}
\begin{split}
\Big|&\myfrac{1}{\vol{A^{}_n}} \bigl(\res{A^{}_{n}}{P} \ast \, \widetilde{P}\bigr) (u)- h(u)\Big| \,\leqslant \, \sum_{\substack{k,\ell = 1 \\ k \neq \ell}}^N \left|c^{}_k \ts \overline{c^{}_\ell} \right| \left|\myfrac{1}{\vol{A^{}_n}} \int_{A^{}_n} \ee^{2\pi \ii s\cdot (y^{}_k- y^{}_\ell)}\, \dd s\right| \ts .
\end{split}
\end{equation}
This upper bound does not depend on $u$, and since 
\begin{equation*}
    \lim_{n\to \infty} \sum_{\substack{k,\ell = 1 \\ k \neq \ell}}^N \left|c^{}_k \ts \overline{c^{}_\ell} \right| \left|\myfrac{1}{\vol{A^{}_n}} \int_{A^{}_n} \ee^{2\pi \ii s\cdot (y^{}_k- y^{}_\ell)}\, \dd s\right| \, =\,  0 
\end{equation*}
by a direct computation, (a) holds.  The proof of (b) is similar.

Next, to prove (c), let $K\subseteq \RR^d$ be compact. For any $u \in K$, we have
\begin{equation*}\label{eq:P^{}_N-uniform-triangle}
\begin{split}
\Big|\myfrac{1}{\vol{A^{}_n}} &\bigl( \res{A^{}_{n}}{P} \ast \widetilde{\,\res{A^{}_{n}}{P}} \bigr)(u) - h(u) \Big| \\[2pt] 
& \quad  \, \leqslant  \, \myfrac{1}{\vol{A^{}_n}} \left|\bigl(\res{A^{}_{n}}{P} \ast \widetilde{\,\res{A^{}_{n}}{P}} \bigr)(u) - \bigl(\res{A^{}_{n}}{P} \ast \, \widetilde{P} \bigr) (u)\right| + \Big|\myfrac{1}{\vol{A^{}_n}} \bigl(\res{A^{}_{n}}{P} \ast \,\widetilde{P}\bigr) (u) - h(u)\Big|\ts .
\end{split}
\end{equation*}
By (a), it suffices to show that the first term converges uniformly on $K$ to $0$. To do so, we employ the following upper bound:
\begin{equation}\label{eq:P^{}_N-uniform-est-2}
\begin{split}
\myfrac{1}{\vol{A^{}_n}} & \left|\bigl(\res{A^{}_{n}}{P} \ast \widetilde{\,\res{A^{}_{n}}{P}} \bigr) (u) - \bigl(\res{A^{}_{n}}{P} \ast \, \widetilde{P} \bigr) (u)\right|  \\ 
& \,= \, \myfrac{1}{\vol{A^{}_n}} \left|\int_{\RR^d}P(s)  \widetilde{P(s-u)} \boldsymbol{1}^{}_{A^{}_{n}}(s)\boldsymbol{1}^{}_{A^{}_{n}}(s-u)  \, \dd s  - \int_{\RR^d}P(s)\widetilde{P(s-u)} \boldsymbol{1}^{}_{A^{}_{n}}(s) \dd s\right| \\[2pt] 
&\,\leqslant  \, \|P\|^{2}_\infty \cdot \myfrac{1}{\vol{A^{}_n}} \int_{\RR^d} \bigl|\boldsymbol{1}^{}_{A^{}_{n}}(s)\cdot (\boldsymbol{1}^{}_{A^{}_{n}}(s-u)-1)\bigr| \,\dd s \\[2pt] 
&\,= \, \|P\|^{2}_\infty \cdot \myfrac{\vol(A^{}_n \backslash (u+A^{}_n))}{\vol(A^{}_n)} \,\leqslant \, \|P\|^{2}_\infty \cdot \myfrac{\vol(\partial^K(A^{}_n))}{A^{}_n} \ts ,
\end{split}
\end{equation}
for all $n \in \NN$. Now, we profit from the van Hove property and obtain
\begin{equation*}
    \lim_{n\to \infty} \|P\|^{2}_\infty \cdot \myfrac{\vol(\partial^K(A^{}_n))}{A^{}_n} = 0 \ts .
\end{equation*}
Observe that the estimate in Eq.~\eqref{eq:P^{}_N-uniform-est-2} does not depend on $u \in K$. This implies uniform convergence on $K$.  Since $K$ was arbitrary, the claim (c) holds.

Lastly, it follows from (c) and Lemma~\ref{lem:uniform-convergence-compact} that $\gamma^{}_\mu$ exists along $\cA$ for $\mu = P\lambda$ and we have $\gamma^{}_\mu = h\lambda$.
\end{proof}

\begin{remark}
The proof of Lemma~\ref{lem:P^{}_N-uniform} also shows that (a) holds for any F{\o}lner sequence, whereas (b) requires the van Hove property. Moreover, by a limiting argument, one can also show that this result extends to the Bohr almost periodic functions, as any such function can be uniformly approximated by a~sequence of trigonometric polynomials \cite[Thm.~4.3.5(iv)]{MS2017}. \exend
\end{remark}

\pagebreak 
\begin{theorem}
\label{thm:Bap}
    Let $\cA = (A^{}_n)^{}_{n\in\NN}$ be a van Hove sequence and let $f\in \Bap^{\, 2}(\RR^d)$. 
    Then, for any sequence of trigonometric polynomials $(P^{}_N)^{}_{N\in\NN}$ with $\bigl\|f-P^{}_N\bigr\|^{}_{b,2,\cA} \longrightarrow 0$, the following statements hold.
    \begin{enumerate}[label=(\alph*)]
    \item There exists a unique function $h \in \SAP(\RR^d)$ such that 
    \[
        \bigl\|\lb P^{}_{N},P^{}_{N} \rb^{}_{\cA} - h \bigr\|^{}_\infty\,  \longrightarrow \, 0  \ts .
    \] 
    \item If $(Q^{}_N)^{}_{N\in\NN}$ is any other sequence of trigonometric polynomials satisfying $\bigl\|f - Q_N\bigr\|^{}_{b,2,\cA} \longrightarrow 0$ as $N\rightarrow \infty$, then 
    \[
        \bigl\|\lb Q_N, Q_N \rb^{}_{\cA} - h\bigr\|^{}_\infty \, \longrightarrow \, 0  \ts .
    \]
    \item The sequence $\frac{1}{\vol(A^{}_n)} \res{A^{}_{n}}{f} \ast \, \widetilde{\res{A^{}_{n}}{f}}$ converges
    uniformly on compact sets to $h$.
    \end{enumerate}
    In particular, the autocorrelation $\gamma^{}_\mu$ for $\mu = f\lambda$ exists 
    along $\cA$ and 
    \[
        \gamma^{}_\mu \, = \, h\lambda \ts.
    \]
\end{theorem}

\begin{proof}
To prove (a), it suffices to show that $\lb P^{}_N, P^{}_N \rb^{}_{\cA}$ is Cauchy in $(\Cu(\RR^d), \|\cdot\|^{}_\infty)$, since uniform limits of trigonometric polynomials are always Bohr almost periodic (see \cite[Thm.~4.3.5]{MS2017}). By Proposition~\ref{prop:CS}, for any $N,M$, we obtain
\begin{align*}
\bigl\|\lb P^{}_N, P^{}_N \rb^{}_{\cA} - \lb P^{}_M, P^{}_M \rb^{}_{\cA}\bigr\|^{}_\infty &\,\leqslant \, \bigl(\|P^{}_N\|^{}_{b,2,\cA} + \|P^{}_M\|^{}_{b,2,\cA}\bigr) \ts \|P^{}_N - P^{}_M\|^{}_{b,2,\cA}\,.
\end{align*}
By assumption, $\|f-P^{}_N\|^{}_{b,2,\cA} \longrightarrow 0$, so $P^{}_N$ is Cauchy in the Besicovitch~$2$-norm.  Thus, $h^{}_{N} = \lb P^{}_N, P^{}_N \rb^{}_{\cA}$ is Cauchy with respect to $\|\cdot\|^{}_\infty$. Moreover, for all $N$, we have $h^{}_{N} \in \SAP(\RR^d)$, which is a closed subspace in $(\Cu(\RR^d), \|\cdot\|^{}_\infty)$. Therefore, there exists $h \in \SAP(\RR^d)$ such that $\|h^{}_N - h\|^{}_\infty \longrightarrow 0$ as $N\longrightarrow \infty$.

To prove (b), assume that $(Q^{}_N)^{}_{N\in\NN}$ is another sequence such that $\|f -  Q^{}_N\|^{}_{b,2,\cA} \longrightarrow 0$.  By (a), there exists a unique function $g \in \SAP(\RR^d)$ such that
\[
    \bigl\|\lb Q^{}_N, Q^{}_N \rb - g\bigr\|^{}_\infty \, \longrightarrow \,  0 \ts .
\]
Using triangle inequality and \eqref{eq:CS}, we have, for all $N \in \NN$, \begin{align*}
\|g - h\|^{}_{\infty} & \leqslant \, \bigl\|g - \lb Q^{}_N, Q^{}_N \rb^{}_{\cA} \bigr\|^{}_{\infty} + \bigl\|\lb Q^{}_N{-}P^{}_N, Q^{}_N \rb^{}_{\cA}\bigr\|^{}_{\infty}  + \bigl\|\lb P^{}_N, Q^{}_N{-}P^{}_N \rb^{}_{\cA}\bigr\|^{}_{\infty} + \bigl\|\lb P^{}_N, P^{}_N \rb^{}_{\cA} - h \bigr\|^{}_{\infty} \\[3pt]
& \leqslant \,  \bigl\|g - \lb Q^{}_N, Q^{}_N \rb^{}_{\cA} \bigr\|^{}_{\infty} + \bigl\|\lb P^{}_N, P^{}_N \rb^{}_{\cA} - h \bigr\|^{}_{\infty}\\[2pt]
& \qquad \qquad \qquad 
+\bigl(\|P^{}_N\|^{}_{b,2,\cA} + \|Q^{}_N\|^{}_{b,2,\cA}\bigr) \ts\cdot \ts  \bigl(\|Q^{}_N - f\|^{}_{b,2,\cA} + \|f - P^{}_N\|^{}_{b,2,\cA}\bigr)\,.  
\end{align*} 
Taking $N\rightarrow \infty$ shows that $\| g - h \|^{}_\infty = 0$, so $g = h$.  Therefore, (b) holds.

To prove (c), let set $K\subseteq \RR^d$ be compact. We use the usual 3 epsilon argument. Indeed, for all $\varepsilon >0$ and for all $u\in K$, one can choose $N\in\NN$ such that for all  $n\in\NN$ sufficiently large, we have
\begin{align}\label{eq:key-inequality}
\Big|\myfrac{1}{\vol(A^{}_{n})} & \bigl(\res{A^{}_{n}}{f} \ast \, \widetilde{\res{A^{}_{n}}{f}}\bigr)(u) - h(u)\Big| \nonumber\\ 
& \, \leqslant  \, \Big|\myfrac{1}{\vol(A^{}_{n})} \bigl( \res{A^{}_{n}}{f} \ast \, \widetilde{\res{A^{}_{n}}{f}} \bigr) (u) - \myfrac{1}{\vol(A^{}_{n})}\bigl(\res{A^{}_{n}}{P^{}_N} \ast \, \widetilde{\res{A^{}_{n}}{P^{}_N}}\bigr) (u) \Big| \\ 
& \quad \quad \quad  +\Big|\myfrac{1}{\vol(A^{}_{n})}\bigl(\res{A^{}_{n}}{P^{}_N} \ast \, \widetilde{\res{A^{}_{n}}{P^{}_N}}\bigr)(u) - h_N(u)\Big| + |h^{}_N(u) - h(u)|  \, < \, \eps \ts.\nonumber
\end{align} 
For the first term, we use the convergence of $P^{}_N$ with respect to the Besicovitch 2-seminorm together with estimate 
\begin{equation*}\label{eq:CS-estimate}
\begin{split}
\myfrac{1}{\vol(A^{}_{n})}&\Big| \bigl(\res{A^{}_{n}}{f} \ast \, \widetilde{\res{A^{}_{n}}{f}} \bigr)(u) - \bigl(\res{A^{}_{n}}{P^{}_N} \ast \, \widetilde{\res{A^{}_{n}}{P^{}_N}} \bigr) (u) \Big| \\[2pt] 
& \leqslant  \bigg( \myfrac{1}{\vol(A^{}_{n})} \int^{}_{A^{}_n} \bigl|(f{-}P^{}_{N})(s) \bigr|^2 \dd s\!\bigg)^{\frac{1}{2}} \lt[\!\bigg( \myfrac{1}{\vol(A^{}_{n})} \int^{}_{A^{}_n} \bigl|f(s)\bigr|^2 \dd s\!\bigg)^{\frac{1}{2}}\! + \bigg( \myfrac{1}{\vol(A^{}_{n})} \int^{}_{A^{}_n} \bigl|P^{}_N(s) \bigr|^2 \dd s\!\bigg)^{\frac{1}{2}} \rt] \ts,
\end{split}
\end{equation*}
which can be obtained using the usual Cauchy--Schwarz inequality and a change of variables. The second term in \eqref{eq:key-inequality} can be made arbitrarily small on $K$ due to Lemma~\ref{lem:P^{}_N-uniform}, and the last term can also be made arbitrarily small due to the convergence of $h^{}_{N}$. As $K$ was arbitrary, this shows that $\frac{1}{\vol(A^{}_{n})}\res{A^{}_{n}}{f} \ast \, \widetilde{\res{A^{}_{n}}{f}}$ converges uniformly on compact sets to the function $h$, so (c) holds. A detailed proof with all the necessary steps can be found in \cite{Kor2026}. 
Lastly, it follows from (c) and Lemma~\ref{lem:uniform-convergence-compact} that $\gamma^{}_\mu$  for $\mu = f\lambda$ exists along $\cA$ and we have $\gamma^{}_\mu = h\lambda$.
\end{proof}

\begin{remark}
In the case when $f \in L^\infty(\RR^d)$, we get that $h = \eta^{}_{f}$ by 
Proposition~\ref{prop:bounded-conv}. Therefore, we recover the result from Proposition~\ref{prop:autocorrelation-reflected} which states that
\[
    \gamma^{}_\mu \, =\,  \eta^{}_{f} \lambda \,,
\]
when $\mu = f \lambda$ is translation bounded, and the result from Corollary~\ref{cor:Besicovitch-ac} which states that $\gamma^{}_{P^{}_N}$ converges uniformly to $\eta^{}_{f}$. \exend
\end{remark}

From this point onward, we use the following definition of autocorrelation for Besicovitch $2$-almost periodic functions, which is justified by the previous theorem. 
\begin{definition}
Let $f \in \Bap^2(\RR^d)$. We write $\eta^{}_f$ to denote the unique function $h \in \SAP(\RR^d)$  such that the autocorrelation $\gamma^{}_\mu$ for the measure $\mu = f \lambda$ satisfies $\gamma^{}_\mu = h \lambda$.  Moreover, since $\eta^{}_f$ is bounded, it defines a tempered distribution, and we denote its Fourier transform by $\widehat{\eta^{}_f}$.
\end{definition}

\noindent Since the measure $\gamma^{}_\mu = \eta^{}_f \lambda$ is positive definite by construction, we get that $\gamma^{}_\mu$ is also Fourier transformable as a measure, with positive Fourier transform, by \cite[Thm.~4.5]{BF1975}.  
Moreover, the Fourier transform of $\gamma_\mu$ as a measure coincides with the Fourier transform of $\eta^{}_f$ as a tempered distribution by \cite[Thm.~5.2]{Str2020}. See also  \cite{SS2019} for explicit examples.

\section{Plane waves and Besicovitch almost periodic functions}
\label{sec:plane}
We can explicitly treat the case of $n$-dimensional plane waves. Indeed, consider $f(x) = \ee^{2\pi\ii a\cdot x}$ and $g(x) = \ee^{2\pi\ii b\cdot x}$ with $a,b \in \RR^d$. Throughout this section, we restrict to the standard van Hove sequence of symmetric (hyper)cubes  $\cA = \bigl([-R,R]^n\bigr)^{}_{R > 0}$.  This is the assumed van Hove sequence, unless stated otherwise.

Then, for the reflected Eberlein convolution of $f$ and $g$, one has
\begin{align*}
\lb f,g \rb (y) & = \, \lim_{R \to \infty} \myfrac{1}{(2R)^n} \int_{[-R,R]^n} \ee^{2\pi\ii a\cdot x} \ee^{-2\pi\ii b\cdot (x-y)} \ts \dd x \, = \, \ee^{2\pi\ii b\cdot y} \lim_{R \to \infty} \prod_{j=1}^n \mathrm{sinc}\bigl(2\pi R(a^{}_{j} - b^{}_j) \bigr)
 \\
& = \, \begin{cases}
   \ee^{2\pi\ii b\cdot y}, & \mbox{if} \ a=b, \\
   0, & \mbox{otherwise.}
\end{cases}
\end{align*}

This implies that the autocorrelation measure of a plane wave described by $f(x) = \ee^{2\pi\ii a\cdot x}$ is absolutely continuous and reads
$ \eta^{}_{f} \, =\,  f$.
By standard arguments, it follows that the result does not change if we consider centred balls instead of cubes. 

Furthermore, we can employ general properties of the reflected Eberlein convolution (see Proposition~\ref{prop:linear}) to obtain the following result for any trigonometric polynomial~$Q$.
\begin{lemma}
    Let $Q(x)$ be a trigonometric polynomial in $n$ variables, i.e., 
    \[Q(x) \, = \, \sum_{\ell =1}^N c^{}_{\ell} \ee^{2\pi\ii d^{}_{\ell} \cdot x} \qquad \mbox{with } x\in\RR^d \]
    for some $N\in \NN$, with $c^{}_{\ell} \in \CC$ and $d^{}_{\ell} \in \RR^d$ for all $\ell$. Then, the natural autocorrelation exists, is absolutely continuous with respect to Lebesgue measure, and its density function reads
    \[ \eta^{}_{Q}(x) \, = \, \sum_{\ell=1}^N |c^{}_{\ell}|^2 \ee^{2\pi\ii d^{}_{\ell} \cdot x}.  \]
    Its diffraction measure reads
    \[ \widehat{\ts \eta^{}_{Q}\ts} \, =\, \sum_{\ell=1}^N |c^{}_{\ell}|^2 \ts \delta^{}_{d^{}_{\ell}} \ts. \]
\end{lemma}
\begin{proof}
    The proof is easy and follows from the (sesqui-)linearity of the reflected Eberlein convolution, as
    \[ \eta^{}_{Q}(x) \, = \, \lb Q, Q \rb(x) \, = \, \sum_{k,\ell=1}^N  c^{}_{k}\overline{c^{}_{\ell}} \, \lb \ee^{2\pi\ii d^{}_{k}}, \ee^{2\pi\ii d^{}_{\ell}}\rb \, =\, \sum_{\ell=1}^N |c^{}_{\ell}|^2 \ee^{2\pi\ii d^{}_{\ell} \cdot x}. \]
    The rest follows from the linearity of the Fourier transform. 
\end{proof}

\begin{remark}
\label{rem:plane}
As we shall see later, the symmetric choice 
$Q(x) = \frac{1}{\sqrt{2}} \bigl( \ee^{2 \pi \ii r \cdot x} + \ee^{-2 \pi\ii r x}\bigr) = \sqrt{2} \cos (2 \pi r \cdot x)$ is especially interesting.  When $r\ne 0$, this
leads to the autocorrelation $\eta (x) = \cos (2 \pi r \cdot x)$ and the
diffraction measure $\widehat{\eta} = \frac{1}{2} (\delta^{}_{r} + \delta^{}_{-r})$. In one dimension, this represents the uniform distribution measure on $\{ -r, r \}$, in other words $\theta^{}_{r}$ as defined in Definition \ref{def:spher_meas}. \exend
\end{remark}

We now move to the realm of Besicovitch almost periodic functions, which can be understood as functions that can be approximated with arbitrary precision by trigonometric polynomials (with respect to the Besicovitch seminorm!). Recall from Theorem~\ref{thm:hilbert} that $f \in \Bap^{\, 2}(\RR^d)$ if and only if $f$ is given by
\begin{equation}
\label{eq:f_expansion}
f(x) \, = \, \sum_{m=1}^\infty c^{}_m \ee^{2 \pi \ii a^{}_m \cdot x} 
\ts ,
\end{equation}
for a sequence $(c^{}_m)^{}_{m\in \NN}$ in $\ell^{2}_{}$ and points $a^{}_m$ in $\RR^d$.  Then
\begin{equation}
\label{eq:P^{}_{N}}
P^{}_N(x) \, = \, \sum_{m=1}^N c^{}_m \ee^{2 \pi \ii a^{}_m \cdot x} \ts ,
\end{equation}
is a sequence of trigonometric polynomials with $\bigl\|f-P^{}_N\bigr\|^{}_{b,2} \, \longrightarrow \, 0$, and by Proposition \ref{thm:Bap}, we have that
\[
\eta^{}_f(x) \, =\,  \sum_{m=1}^\infty |c^{}_m|^2 \ts \ee^{2 \pi \ii a^{}_m \cdot x} \, ,
\]
belongs to $\SAP(\RR^d)$.  Thus, we obtain the following result. The proof we present uses abstract results, profiting from the key observation that $\eta^{}_{f}\in SAP(\RR^d)$.

\begin{theorem}\label{thm:Besicovitch-diffraction}
    Let $f\in \Bap^{\, 2}(\RR^d)$ and suppose that $f$ can be written as in Eq.~\eqref{eq:f_expansion}. Then, $\eta^{}_{f}$ exists and the diffraction reads
    \[ \widehat{\eta^{}_{f}} \, = \, \sum_{m=1}^{\infty} |c^{}_{m}|^2_{} \, \delta^{}_{a^{}_{m}} \ts .\]
\end{theorem}
\begin{proof}
Since $f\in \Bap^{\, 2}(\RR^d)$, the autocorrelation $\eta^{}_{f}$ exists and is strongly almost periodic by Theorem~\ref{thm:Bap}. It follows that the diffraction $\widehat{\eta^{}_f}$ is pure point \cite[Cor.~4.10.13]{MS2017}. 

To compute the diffraction amplitudes, recall that the amplitude of a Bragg peak of a translation-bounded measure $\mu$ on $\RR^d$ is given by \cite[Thm.~4.10.14]{MS2017} 
\begin{equation}\label{eq:intensity-formula}
    \widehat{\mu}(\{x\}) \, = \, \lim_{n\to \infty} \myfrac{1}{\vol(A^{}_{n})} \int_{A^{}_n} e^{-2\pi i x \cdot y} \,\dd \mu(x) \,.
\end{equation}
Fix $N\in\NN$ and let $(P^{}_{N})_{N\in\NN}$ be the sequence of approximating trigonometric polynomials given in Eq.~\eqref{eq:P^{}_{N}}.  Since both $\eta^{}_f$ and $\eta_{P^{}_{N}}$ are in $\Cu(\RR^d)$, they are translation bounded when interpreted as measures, so we can apply Eq.~\ref{eq:intensity-formula} to obtain
\begin{align*} 
\bigl|\widehat{\eta^{}_{f}} (\{x\}) - \widehat{\eta^{}_{P^{}_N}} (\{x\}) \bigr|  \, &= \, \Big| \lim_{n\to \infty} \myfrac{1}{\vol(A^{}_{n})} \int_{A^{}_{n}}  \ee^{-2\pi \ii x \cdot y} (\eta^{}_{f}(y) - \eta^{}_{P^{}_{N}}(y)) \,\dd y \Big|  \\ 
&\leqslant \, \lim_{n\to \infty} \myfrac{1}{\vol(A^{}_{n})} \int_{A^{}_{n}} |\eta^{}_{f}(y) - \eta^{}_{P^{}_{N}}(y)| \,\dd y\\ 
&\leqslant  \, \|\eta^{}_{f} - \eta^{}_{P^{}_{N}}\|^{}_{\infty} \ts .
\end{align*}
Taking $N\rightarrow \infty$, it follows from Theorem~\ref{thm:Bap} that
\[\widehat{\eta^{}_{f}} (\{x\}) \, = \, \lim_{N\to \infty} \widehat{\eta^{}_{P^{}_N}} (\{x\}). \] 
Finally, observe that when $x \neq a^{}_m$ for any $m$, we have $\widehat{\eta^{}_{P^{}_N}}(\{x\}) \,= \,0$
for all $N$, so $\widehat{\eta^{}_{f}} (\{x\}) \, = \, 0$.  On the other hand, if $x = a^{}_m$, we have $\widehat{\eta^{}_{P^{}_N}}(\{x\}) \,= \, |c^{}_m|^2$
for all sufficiently large $N$, so $\widehat{\eta^{}_{f}} (\{a^{}_m\}) \, = \, |c^{}_m|^2$.  Therefore, the claim holds.
\end{proof}

\begin{remark}
\label{rem:uniqueness_plane}
    It is worth mentioning that this proof, read backwards, provides an explicit way of constructing functions with a prescribed pure-point diffraction measure, once the phases are specified. Indeed, pick any positive finite pure-point measure $\mu$ on $\RR^d$ (the support can be \emph{any} countable subset of $\RR^d$) and choose \emph{any} phases $\xi^{}_{m}$. Then, there exists a unique function $f\in \Bap^{\, 2}(\RR^d)$ which will always be of the form \eqref{eq:f_expansion} with $a^{}_{m}$ being points from the support and $c^{}_{m} = \xi^{}_m \sqrt{\mu(\{a^{}_{m}\})}$. Uniqueness follows from \cite[Cor.~3.20]{LSS2020}. \exend
\end{remark}
It may seem natural to weaken the assumption that $f \in \Bap^2(\RR^d)$; however, attempting to do so for $f\in BL^{2}(\RR^d)$ leads to subtle and technical challenges \cite{LS2025}. 

\section{Spherical waves and radially concentrated diffraction}\label{sec:spherical}
One of the main results of Theorem~\ref{thm:Besicovitch-diffraction} was that the natural autocorrelation of Besicovitch $2$-almost periodic functions exist, and that the corresponding diffraction is pure point. In this section, we will prove an analogous result for spherical waves. To do so, we develop a broader framework for working with radially symmetrical measures. 
Throughout, we restrict ourselves to the standard van Hove sequence of closed balls $\cA = (B^{}_R)^{}_{R > 0}$. This is the assumed van Hove sequence for this section, unless stated otherwise.

Let us begin by introducing the notion of a \emph{radially concentrated measure}, which provides a radial counterpart to pure point measures. 

\begin{definition}
We say that a measure $\mu$ is \emph{radially concentrated} if $\mu$ is concentrated on $\bigcup_{r\in A}\Sph^{d-1}_r$ for some countable set $A\subseteq [0,\infty)$.
\end{definition}

\begin{remark}\label{rem:ps-sum}
A measure $\mu$ is  \emph{radially concentrated}  if and only if 
\[
\mu \, =\,  \sum_{r\in A} \mu_r\ts ,
\]
for some countable set $A\subseteq [0,\infty)$, where $\mu_r$ is concentrated on $\Sph^{d-1}_r$ for each ${r\in A}$. Moreover, if $\mu$ is non-zero, then this decomposition can be written uniquely with each $\mu^{}_r = \mu\bigl|^{}_{\Sph^{d-1}_r}$ non-zero. \exend
\end{remark}

\begin{example}
Consider a lattice Dirac comb $\delta^{}_{\ZZ^d}$. This is a pure-point measure. It is also a radially concentrated measure with the countable set $A$ being a subset of $\{ \sqrt{k} \ : \ k\in \NN^{}_{0} \}$ and with $\mu^{}_{k}$ being a~sum of $r^{}_{d}(k)$ Dirac measures, where $r^{}_{d}(k)$ is the number of ways of writing $k$ as a sum of $d$ squares. 
This example shows that 
\[ \mbox{radially concentrated} \ \notimplies \ \mbox{radially symmetric.} \]

Of course, we also have that radially symmetric need not imply radially concentrated; indeed, \emph{any} nonzero measure $\mu = f \lambda$ with radially symmetric density function $f$, even with $f = 1$, is radially symmetric with $|\mu|(\Sph^{d-1}_r) =0$ for all $r$. \exend
\end{example}

Similarly to how a measure has a unique decomposition as the sum of a pure-point measure and a~continuous measure, we can decompose a measure uniquely as the sum of a radially concentrated measure and what we will call a \emph{radially dispersed measure}.

\begin{definition}
    We say that a measure $\mu$ is \emph{radially dispersed} if $|\mu|(\Sph^{d-1}_r) = 0$ for all $r \in [0,\infty)$.
\end{definition}

\begin{remark}\label{rem:ps-subspace}
    One can easily show that the radially concentrated measures and the radially dispersed measures both form linear subspaces of the space of measures. Moreover, the only measure which is both radially concentrated and radially dispersed is $\mu=0$. To see this, let $\mu = \sum_{r\in A} \mu_r$ be a radially concentrated measure with each $\mu_r$  concentrated on $\Sph^{d-1}_r$.  If $\mu$ is also radially dispersed, then we have
    \[
    |\mu|(\Sph^{d-1}_r) \, =\,  0 \,  \Longrightarrow \, |\mu_r|(\Sph^{d-1}_r) \, =\,  0\,  \Longrightarrow \, \mu_r = 0 \ts ,
    \]
    for each $r \in A$, so $\mu = 0$. \exend
\end{remark}

\begin{proposition}\label{prop:ps-decomposition}
Let $\mu$ be a measure. Then $\mu$ has a unique decomposition as 
\[
\mu \,= \, \mu^{}_{\rc}+\mu^{}_{\rd} \ts ,
\]
where $\mu^{}_{\rc}$ is a radially concentrated measure and $\mu^{}_{\rd}$ is a radially dispersed measure.
\end{proposition}

\begin{proof}
    By Remark~\ref{rem:ps-subspace}, it suffices to show that every positive measure can be written as the sum of a~positive radially concentrated measure and a positive radially dispersed measure. 

    To this end, let $\mu$ be a positive measure.
    Define $A \, \defeq \, \{r\geqslant  0\,:\,\mu(\Sph^{d-1}_r)>0\}$.
    Then $A$ must be countable. Indeed, we have
    \[
        A \,=\, \bigcup_{R>0}\bigcup_{m \in \NN} \left\{r \in [0,R) \,:\, \mu(\Sph^{d-1}_r) \geqslant \tfrac{1}{m}\right\} \ts , 
    \]
    and for any given $R$ and $m$, we have $\mu(B_R) < \infty$, which implies that each of the sets in the above union is finite.

    Let $B = \bigcup_{r\in A} \Sph^{d-1}_r$. Now, we define the radially concentrated measure
    \[
    \mu^{}_{\rc} \, \defeq \, \mu\bigl|^{}_B = \sum_{r\in A} \mu\bigl|^{}_{\Sph^{d-1}_r} \ts .
    \]
    Then 
    $\mu^{}_{\rd} \, \defeq \,  \mu - \mu^{}_{\rc} = \mu\bigl|^{}_{\RR^d\backslash B}
    $
    is radially dispersed by construction. 
    Since both $\mu^{}_{\rc}$ and $\mu^{}_{\rd}$ are positive, this completes the proof.
\end{proof}

\begin{remark}
Proposition~\ref{prop:ps-decomposition} shows that we have
\[
    \cM(\RR^d) \, = \, \cM^{}_{\rc}(\RR^d) \oplus \cM^{}_{\rd}(\RR^d)\,,
\] 
where $\cM(\RR^d)$ is the space of complex Radon measures, $\cM^{}_{\rc}(\RR^d)$ is the set of the radially concentrated measures, and $\cM^{}_{\rd}(\RR^d)$ is the set of radially dispersed measures. Furthermore, viewing $\cM(\RR^d)$ as a~vector lattice, we get that $\cM^{}_{\rc}(\RR^d)$ is a projection band in $\cM(\RR^d)$ with disjoint complement $\cM^{}_{\rd}(\RR^d)$. \exend
\end{remark}

\begin{remark}\label{rem:sph-singular}
    For any countable set $A\subseteq[0,\infty)$, we have $\lambda\bigl(\bigcup^{}_{r\in A} \Sph^{d-1}_r\bigr) = 0$; hence any radially concentrated measure is singular. Moreover, any pure-point measure is concentrated on a countable union of spheres. Thus, the following implications hold:
    \begin{enumerate}[label=(\alph*)]
        \item $\mu \ \text{pure point} \ \implies \mu \ \text{radially concentrated} \ \implies \mu \ \text{singular}$, and
        \item $\mu \ \text{absolutely continuous} \ \implies \mu \ \text{radially dispersed} \ \implies \mu \ \text{continuous}$.
    \end{enumerate} \exend
\end{remark}

We now consider the simplest and most representative form of radially concentrated diffraction. We will show that the diffraction of a spherical wave consists of a single sphere. We have already seen in Eq.~\eqref{eq:sphere-Bessel} a function whose Fourier transform is the uniform probability measure $\theta_r$ concentrated on the sphere of radius~$r$. In particular, to show that a given measure on $\RR^2$ has diffraction $\theta_r$, it suffices to show that the autocorrelation is the function $J_0(2\pi r \|x \|)$; see Thm.~4.9.28 and Thm.~4.11.5 in \cite{MS2017} for details regarding Fourier inversion of measures. 

For instructive reasons, let us begin with a radial wave in one dimension, described via $f^{}_a(x) = \ee^{2 \pi \ii a | x|}$ with
$a\ne 0$. When $a=0$, one has $\eta (x) \equiv 1$ and
$\widehat{\eta} = \delta^{}_{0}$, in line with
what we already know for plane waves. Thus, assume $a\ne 0$ and set $a=1$. Then, the autocorrelation is given by
\[
    \eta(x) \, =\,  \lim_{L\to\infty} \myfrac{I^{}_{L}(x)}{2L} \quad \mbox{with} \quad  I^{}_{L} (x) \, = \int_{-L}^{L} \ee^{2 \pi \ii | y |} 
    \ee^{- 2 \pi \ii |y-x|} \dd y  \ts .
\]
\noindent An easy computation with case distinctions yields
\[
   I^{}_{L} (x) \, = \, (L-| x |) \ee^{2 \pi \ii | x |}
   + L \ts \ee^{-2 \pi \ii | x |} + 
   \myfrac{\ee^{2 \pi \ii x} - \ee^{-2 \pi \ii x}}{4 \pi \ii} \ts ,
\]
which (including the case $x=0$) means
\[
     \myfrac{I^{}_{L} (x)}{2 L} \, = \, \cos(2 \pi x)
     + \mathcal{O} (L^{-1}) \qquad \text{as $L \to \infty$} \ts ,
\]
and thus $\eta (x) = \cos(2 \pi x)$. 

When $a\ne 1$, a simple scaling argument gives
$\eta(x) = \cos(2 \pi a x)$. Similarly,  $\lb f^{}_a, f^{}_b \rb = 0$ whenever $a\neq b$. We have
thus shown the following.

\begin{proposition}\label{lem:sphere-1}
The natural autocorrelation of the spherical wave in one dimension\/ $f^{}_a(x) = \ee^{2 \pi \ii a | x |}$ with\/
$a\geqslant  0$, exists and is given by\/
$\eta (x) = \cos (2 \pi a x)$. The corresponding diffraction is the
discrete probability measure\/
$\widehat{\eta} = \frac{1}{2} (\delta^{}_{a} +
\delta^{}_{-a})$.  Moreover, $\lb f^{}_a, f_{b} \rb \, = \, 0$ whenever $a\neq b$. \qed
\end{proposition}

\begin{remark}
  It is interesting to compare Proposition~\ref{lem:sphere-1} with
  Remark~\ref{rem:plane}. There, we obtained the autocorrelation
  $\eta (x) = \cos (2 \pi x)$ for the (weighted) superposition of two
  plane waves,
  $\frac{1}{\sqrt{2}}\bigl(\ee^{2 \pi \ii x} + \ee^{-2 \pi \ii
    x}\bigr)$, which is thus homometric with the spherical wave
  $\ee^{2 \pi \ii | x |}$. While this looks surprising at
  first sight, it can be understood via recent results on the diffraction of almost characters \cite{LS2025}. 
  We emphasise that this is a~purely one-dimensional phenomenon, and is no longer true in higher dimensions. \exend
\end{remark}

Next, we consider radial waves in two dimensions, described via $f^{}_a(x) = \ee^{2\pi \ii a \|x\|}$ with $a \neq 0$.  Let us compute the reflected Eberlein convolution for two such functions $f^{}_a, f^{}_b$. Since $f^{}_a$ and $f^{}_b$ are both radial functions, we have that $\lb f^{}_a, f^{}_b \rb$ is also a radial function. Hence, it suffices to consider $\lb f^{}_a, f^{}_b \rb$ evaluated at points $(s,0) \in \RR^2$ for $s \geqslant  0$.  We have
\[
\lb f^{}_a, f^{}_b \rb(s,0) \, = \,  \lim_{R\rightarrow \infty} \myfrac{1}{\pi R^2} \int_{B^{}_R} f^{}_a(x) \overline{f^{}_b(x-(s,0))}\,\dd x \ts .
\]
Converting to polar coordinates and simplifying the integrand, we get 
\[
\lb f^{}_a, f^{}_b \rb(s,0) \, =\,  \lim_{R\rightarrow \infty} \myfrac{1}{\pi R^2}  \int_{-\pi}^\pi \int_0^R \ee^{\ii ar} \ee^{-\ii b\sqrt{r^2 - 2rs\cos\theta + s^2}\,}\ts r \,\dd r \,\dd \theta\ts .
\]

\noindent It is easy to see that
\[
\sqrt{r^2 - 2rs\cos\theta + s^2} 
\, =\,  r\sqrt{1-\tfrac{2s\cos\theta}{r}+\tfrac{s^2}{r^2}} \, \sim \,  r-s\cos\theta  \qquad \mbox{as} \ r \longrightarrow \infty\ts ,
\]
geometrically, or via the Taylor expansion of the function $\sqrt{1-x}$ around zero. Then, using continuity of the exponential function, one can show the following.

\begin{fact}\label{lem:approx}
Let $a\neq 0$. Then, as 
\[\ee^{-\ii a\sqrt{r^2 - 2rs\cos\theta + s^2}} \sim \ee^{- \ii a(r-s\cos\theta)} 
\]
as $r \to \infty$. \exend
\end{fact}

\noindent This fact allows us to replace $\sqrt{r^2 - 2rs\cos\theta + s^2} $ by $r - s\cos\theta$ in the expression for $\lb f^{}_a, f^{}_b \rb$.

\begin{lemma}
\label{lem:2d-sphere-integral} 
Let $a,b \neq 0$ and $s \geqslant 0$.  Then
\[
\lb f^{}_a, f^{}_b \rb \ts (s,0) \,=\, \lim_{R \rightarrow \infty} \myfrac{1}{\pi R^2}\int_{-\pi}^\pi\int_{0}^R \ee^{\ii ar} \ee^{-\ii b(r-s\cos\theta)} \ts r \,\dd r \,\dd\theta\ts .
\] 
\end{lemma}
\begin{proof}
This result can be shown by standard asymptotic arguments, and a proof using analytic arguments can be found in \cite{Kor2026}.
\end{proof}

We can now show that the Bessel function of the first kind gives the autocorrelation of a spherical wave in two dimensions. 
\begin{theorem}\label{thm:main}
The natural autocorrelation of \/ $f^{}_a(x) = \ee^{2 \pi \ii a \| x \|}$ with\/
$a \neq 0$, the spherical wave in two dimensions, exists and is given by
\[
\eta^{}_{f^{}_{a}} (x) = J_0\bigl(2\pi a\|x\|\bigr)\ts .
\]
The corresponding diffraction is $\widehat{\,\eta^{}_{f^{}_{a}}} = \theta^{}_{|a|}$, the uniform probability measure on the sphere of radius $|a|$ in two dimensions.  Moreover, we have $\lb f^{}_a, f^{}_b \rb = 0$ whenever $a\neq b$.
\end{theorem}
\begin{proof}
Let $x \in \RR^2$ and let $s = \|x\|$. By rotational invariance of $\eta^{}_{f}$, Lemma~\ref{lem:2d-sphere-integral} and the integral representation of the Bessel function, we have
\[
\lb f^{}_a, f^{}_b \rb(x)\, =\,  \lim_{R\rightarrow \infty} \myfrac{1}{\pi R^2} \int_{-\pi}^\pi \int_0^R \ee^{2 \pi \ii a r}\ee^{-2\pi \ii b(r-s\cos\theta)} r \,\dd r \,\dd \theta \,  = \, \lim_{R\rightarrow \infty} \myfrac{2}{R^2}J_0(2 \pi bs) \int_{0}^R \ee^{2 \pi \ii (a-b)r}r \,\dd r\, .
\]
Hence, when $a=b$ we get $\eta^{}_{f^{}_{a}}(x) =  \lb f^{}_a, f^{}_a \rb = J_0(2\pi a \|x\|)$.
On the other hand, when $a \neq b$, integration by parts gives us the following
\[
\lb f^{}_a, f^{}_b \rb(x) = \lim_{R\rightarrow \infty} \myfrac{2}{R^2} \ts J_0(2 \pi b s) \lt(\myfrac{ \ee^{2 \pi \ii(a-b)R} R}{2 \pi \ii (a-b)} + \myfrac{\ee^{2 \pi \ii (a-b)R}-1}{4 \pi^2 (a-b)^2}\rt) = 0\ts .
\]
Since $J^{}_{0}$ is an even function, after taking its Fourier transform, one obtains $\widehat{\,\eta^{}_{f^{}_{a}}} = \theta^{}_{|a|}$.
\end{proof}

\begin{corollary}
    The diffraction of a function of the form
    \[
    P(x) \, = \, \sum_{k=1}^n c^{}_k \ee^{2\pi \ii a^{}_k \|x\|} \qquad \forall x \in \RR^2\ts ,
    \]
    where $c^{}_k \in \CC$ and $a^{}_k\neq 0$ for each $k\in \{1, \dots, n\}$, is given by
    \[
    \widehat{\ts \eta^{}_{P}} \, = \, \sum_{k=1}^n |c^{}_k|^2\theta^{}_{|a^{}_k|} \ts .
    \]
\end{corollary}
\begin{proof}
    By Proposition~\ref{prop:linear} and Theorem~\ref{thm:main}, the natural autocorrelation for $P$ and for all $x\in\RR^2$ satisfies
    \begin{align*}
        \eta^{}_{P}(x) &\,=\,  \lb P,P\rb (x)\,  =\,  \raisebox{-1.5pt}{\scalebox{1.5}{\lb}}\sum_{k=1}^n c^{}_k \ee^{2\pi \ii a^{}_k \|x\|},\sum_{\ell=1}^n c^{}_\ell \ee^{2\pi \ii a^{}_\ell \|x\|} \raisebox{-1.5pt}{\scalebox{1.5}{\rb}} \\
        &\, =\, \sum_{k=1}^n\sum_{\ell=1}^n c^{}_k \overline{c^{}_\ell}\, J^{}_0\bigl(2\pi a^{}_k \|x\|\bigr) \ts \delta(k-\ell) \, = \, \sum_{k=1}^n |c^{}_k|^2 J^{}_0\bigl(2\pi a^{}_k \|x\|\bigr)\ts .
    \end{align*}
    By linearity of the Fourier transform, the claim follows.
\end{proof}

By Theorem \ref{thm:orthogonal}, we have that whenever the reflected Eberlein convolution of translation-bounded measures exists and is nonzero, the corresponding diffractions cannot be mutually singular. Applying this fact in the case where one of the measures is the function $f^{}_a$ with diffraction $\theta^{}_{|a|}$, we get the following sufficient condition for the existence of a circle of radius $a$ in the diffraction of a given translation-bounded measure.
\begin{corollary}
\label{cor:orthogonal}
    Let $a\neq0$ and let $f^{}_a(x) = \ee^{2\pi \ii a \|x\|}$ with $x\in\RR^2$. If $\mu$ is a~translation-bounded measure with diffraction $\widehat{\ts \gamma^{}_{\mu}}$ and  \/$\lb\mu,f^{}_a\lambda\rb \neq 0$, then the measures $\widehat{\ts \gamma^{}_{\mu}}$ and $\theta^{}_{|a|}$ are not mutually singular.  In particular, we have
    \[
        \widehat{\ts \gamma^{}_{\mu}}(\mathbb{S}^1_{|a|}) > 0 \ts , 
    \] 
    i.e., the diffraction measure of $\mu$ has a nontrivial component that is concentrated on the circle of radius $|a|$.  Moreover, if $\gamma_\mu$ is radially symmetric, then there exists some $c > 0$ such that $c \ts\ts \theta^{}_{|a|} \ll \widehat{\ts \gamma^{}_{\mu}}$.\qed

\end{corollary}
This corollary may become a helpful tool when studying tilings with statistical circular symmetry, such as the pinwheel tiling \cite{Rad1994} or its various generalisations \cite{Fre2008,Sad1998}. It is known that such tilings have radially symmetric diffraction \cite{Fre2008,Kor2026}; however, there is not a single tiling for which one has been able to show the existence of a circle in its diffraction. This is mostly because computing the autocorrelation measure explicitly is extremely difficult and seemingly impossible. This criterion bypasses the need to determine the autocorrelation explicitly, allowing one to work directly with the underlying structure. For inflation tilings, it is conceivable that one could explicitly profit from the inflation structure of the tiling itself (and not that of the autocorrelation, which is always more complicated) to obtain a solvable, recursive structure. 

Next, we look at the diffraction of a spherical wave in $\RR^d$ for $d\geqslant  3$. As in the circular case, we compute the reflected Eberlein convolution of the \emph{simple spherical waves} $f^{}_a(x) = \ee^{2\pi \ii a \|x\|}$  and $f^{}_b(x) = \ee^{2\pi \ii b \|x\|}$ where $a,b>0$.  By rotational invariance, it is enough to evaluate $\lb f^{}_a,f^{}_b \rb (x)$ for $x = (s, 0, \dots, 0) \in \RR^d$ and with $s \geqslant  0$. Using $d$-dimensional spherical coordinates,
we obtain
\begin{equation}\label{eq:higher-reflected-cov}
\lb f^{}_a, f^{}_b \rb (x) \, =\,  \lim_{R\to \infty} \myfrac{\vG\!\lt(\tfrac{d+2}{2}\rt)\Theta_d^{}}{\pi_{}^{\frac{d}{2}}R_{}^d}\,  \int_0^R \int^{\pi}_0 r^{d-1}  \ee^{2\pi\ii a r} \ee^{-2\pi\ii b \sqrt{r_{}^2+s_{}^2-2rs\cos\theta^{}_1}} \sin^{d-2}\theta^{}_1 \, \dd \theta^{}_1 \, \dd r \ts ,  
\end{equation}
where 
\[
\Theta^{}_d \, =\, \int^{\pi}_0 \sin^{d-3}\theta^{}_2 \dd\ts\theta^{}_2\,  \cdots  \int^{\pi}_0 \sin\theta^{}_{d-2} \dd\ts\theta^{}_{d-2}\int^{2\pi}_0  \dd\ts\theta^{}_{d-1} \, =\, \myfrac{2 \pi^{\frac{d-1}{2}}}{\vG\!\left(\frac{d-1}{2}\right)}\ts .
\]
Observe that we again encounter an exponential function with exponent of the form $\sqrt{r_{}^2+s_{}^2-2rs\cos\theta^{}}$. Thus, we can follow the same approach as for Theorem~\ref{thm:main}. Let us note that one can also use a different strategy, as outlined in \cite{BKM25}. 

Now, by Fact~\ref{lem:approx} and \eqref{eq:Bessel-integral}, we get
\begin{align*}
\lb f^{}_a, f^{}_b \rb (x) &\, =\,  \lim_{R\to \infty} \myfrac{\vG \! \lt(\tfrac{d+2}{2}\rt)\Theta_d^{}}{\pi_{}^{\frac{d}{2}}R_{}^d} \int_0^R \int^{\pi}_0 r^{d-1}  \ee^{2\pi\ii a r} \ee^{-2\pi\ii b (r - s\cos\theta_1^{})} \sin^{d-2}\theta^{}_1 \, \dd \theta_1^{} \, \dd r  \\[3pt]
&\, =\, \lim_{R\to \infty} \myfrac{\vG \! \lt(\tfrac{d}{2}\rt) d}{(\pi_{}bs)^{\frac{d}{2}-1}R_{}^d} J_{\frac{d}{2}-1}(2\pi b s) \int_0^R r^{d-1}  \ee^{2\pi\ii (a-b) r}   \, \dd r \ts ,  \\
\end{align*}
so when $a=b$ we get
\[
    \eta^{}_{f^{}_{a}}(x) \, =\,  \lb f^{}_a, f^{}_a \rb (x) \, =\,  \lim_{R\to \infty} \myfrac{\vG \! \lt(\tfrac{d}{2}\rt)d}{(\pi_{}as)^{\frac{d}{2}-1}R_{}^d} J_{\frac{d}{2}-1}(2\pi as) \int_0^R r^{d-1}    \, \dd r \, = \, 
    \myfrac{\vG \! \lt(\tfrac{d}{2}\rt)}{(\pi_{}as)^{\frac{d}{2}-1}} J_{\frac{d}{2}-1}(2\pi a s) \ts ,
\]
and when $a\neq b$ it is easy to see via integration by parts that
\begin{equation}
\label{eq:R_n-1}
\int_0^R r^{d-1}  \ee^{2\pi\ii (a-b) r} \, \dd r = \cO(R^{d-1}) \qquad \text{as} \ R \longrightarrow \infty\ts ,
\end{equation}
so $\lb f^{}_a, f^{}_b \rb (x) = 0$, as expected.

Moreover, the calculation remains true for $a<0$ or $b < 0$. The only small difficulty is the evaluation of a~Bessel function for a negative real argument. This can be easily bypassed using the series expansion~\eqref{eq:Bessel-series}. Indeed, 
\begin{align*}
\frac{J^{}_{\frac{d}{2}{-}1}\!\bigl( -2 \pi a \| x \| \bigr)}
    {\bigl(- \pi a \| x \| \bigr)^{\frac{d}{2} - 1}} & \, = \, \frac{ \Bigl(\frac{-2 \pi a \| x \|}{2} \Bigr)^{\!\frac{d}{2}{-}1}}{\bigl(- \pi a \| x \| \bigr)^{\!\frac{d}{2} {-} 1}} \!\sum_{m=0}^{\infty}
   \myfrac{(-1)^{m} \bigl(\tfrac{ -2 \pi a \| x \|}{2}\bigr)^{2m}}{m! \, \vG(\frac{d}{2}{+}m)}  \, = \, \sum_{m=0}^{\infty} \!
   \myfrac{(-1)^{m} \bigl( \pi a \| x \|\bigr)^{2m} }{m! \, \vG(\frac{d}{2}{+}m)}  \, =\,  \frac{J^{}_{\frac{d}{2}{-}1}\!\bigl( 2 \pi a \| x \|\bigr)}
    {\bigl(\pi a \| x \| \bigr)^{\frac{d}{2} - 1}} \, .
\end{align*}

\noindent We can summarise this derivation in the following theorem. 

\begin{theorem}\label{thm:sphere-d}
The natural autocorrelation of a $d$-dimensional spherical wave,
$f^{}_a(x) = \ee^{2 \pi \ii a \| x \|}$ with\/
$a \neq 0$, exists and is given by 
\[
    \eta^{}_{f^{}_a}(x) \, =\,  \vG \! \left( \myfrac{d}{2}\right) \, 
    \frac{J^{}_{\frac{d}{2}-1}\bigl( 2 \pi \| ax \|\bigr)}
    {\bigl( \pi \| ax \| \bigr)^{\frac{d}{2} - 1}} \ts.
\]
The corresponding diffraction  is\/
$\widehat{\ts \eta^{}_{f^{}_a}} = \theta^{}_{|a|}$, the uniform probability measure on the sphere of radius $|a|$ in $d$ dimensions. In the limit as\/ $a\to 0$, this converges to the diffraction of
the constant function\/ $1$, with\/ $\eta^{}_{f^{}_0} = 1$ and \/$\widehat{\ts \gamma^{}_{f^{}_0}} = \delta^{}_{0}$.   

Moreover, we have $\lb f^{}_a, f^{}_b \rb = 0$ whenever $a\neq b$. \qed
\end{theorem}

\begin{remark}
\label{rem:spherical_uniq_problem}
This theorem highlights a difference between the plane and spherical wave settings. If one considers a plane wave $\ee^{2\pi \ii a\ts x}$ and its complex conjugate $\ee^{-2\pi \ii a\ts x}$, the diffraction can distinguish them. This is no longer true for any simple spherical wave, as the diffraction image of both consists of one circle of the same radius. This implies that, unlike plane waves, which can be reconstructed from the diffraction up to a phase factor, ambiguity arises for a single spherical wave even without considering different phase factors.  \exend
\end{remark}

Theorem~\ref{thm:sphere-d} establishes the existence of functions with rotationally invariant diffraction using a~different approach than in \cite{BKM25}.  For a function $f$, by \emph{rotationally invariant}, we mean that $f(x) = f(Rx)$ for all $x \in \RR^d$ and all $R \in \SO(d)$.
To analyse such functions in more detail and to provide analogous statements to those in the plane wave setting, we need a suitable functional-analytic framework. 

We start by relating the Besicovitch $p$-norm of a spherically symmetric function to the  Besicovitch $p$-norm of its radial part. This will eventually lead to the subsequent definition of $p$-Besicovitch radially almost periodic functions. 

\begin{proposition}
\label{prop:radial_ineq}
    Let $1 \leqslant  p < \infty$ and let $\cA_1 = \bigl([-R,R]\bigr)_{R>0}$.  Given an even function $f \in BL_{\cA_1}^p(\RR)$, define $F:\RR^d\rightarrow\CC$ by $F(x) \deq f\bigl(\|x\|\bigr)$ for all $x\in\RR^d$.  Then the Besicovitch $p$-seminorm of $F$ satisfies
    \[
        \bigl\|F\bigr\|^{}_{b,p} \, \leqslant \, d^{\frac{1}{p}} \bigl\|f\bigr\|^{}_{b,p,\cA_1}\ts .
    \]
    In particular, we have $F\in BL_{}^p(\RR^d)$.
\end{proposition}

\begin{proof}
    By definition, we have
    \[
        \bigl\|F\bigr\|^{}_{b,p} \, =\,  \limsup_{R\rightarrow\infty} \lt(\myfrac{1}{\vol(B^{}_R)} \int_{B^{}_R} \bigl|f(\|x\|)\bigr|^p \,\dd x\rt)^{\frac{1}{p}} \ts .
    \]
    Converting to higher-dimensional spherical coordinates and utilising the surface area and volume formulae for spheres in $\RR^d$, we obtain
    \begin{align*}
        \bigl\|F\bigr\|^{}_{b,p} &\, = \, \limsup_{R\rightarrow\infty} \lt(\frac{\vG \! \lt(\frac{d}{2}+1\rt)}{\pi^{\frac{d}{2}}R^d} \cdot \frac{2\pi^{\frac{d}{2}}}{\vG \! \lt(
    \frac{d}{2}\rt)} \int_0^R |f(r)|^p r^{d-1}\,\dd r\rt)^{\frac{1}{p}}  \\
    &\, = \, \limsup_{R\rightarrow\infty} \lt(\frac{d}{2R} \int_{-R}^R |f(r)|^p \lt(\frac{|r|}{R}\rt)^{d-1}\,\dd r\rt)^{\frac{1}{p}}  \\
    & \, \leqslant \, \limsup_{R\rightarrow\infty} \lt(\frac{d}{2R} \int_{-R}^R |f(r)|^p\,\dd r\rt)^{\frac{1}{p}}  \, = \, d^{\frac{1}{p}}\bigl\|f\bigr\|^{}_{b,p,\cA_1} \ts . \qedhere
    \end{align*}
\end{proof}

Recall that the usual Besicovitch $2$-almost periodic functions in $\RR^d$ are limits of trigonometric polynomials in the space $BL^{2}_{}(\RR^d)$. For rotationally invariant functions, this setting is unnatural, so we replace the usual trigonometrical polynomials with \emph{simple spherical waves} as introduced earlier and their linear combinations, which we call \emph{spherical waves}. The latter are functions of the form 
\[
\sum_{k=1}^N c^{}_{k} \ee^{2\pi\ii a^{}_{k} \|x\|}\ts ,
\]
where $c^{}_{k} \in \CC$, $a^{}_{k}\neq0$ and $N\in \NN$. 

\begin{definition}
    A spherically symmetric function $F:\RR^d \to \CC$ is called \emph{Besicovitch radially almost periodic}, if there exists a sequence of spherical waves $(S^{}_{N})^{}_{N\in \NN}$, with $S^{}_{N}(x) = \sum_{k=1}^N c^{}_{k} \ee^{2\pi\ii a^{}_{k} \|x\|}$, such that 
    \[ \| F - S^{}_{N} \|^{}_{b,2} \, \longrightarrow \, 0 \quad \mbox{as } N\to \infty.\]
    The space of all such functions will be denoted by $\Brap_{}^{\,2}(\RR^d)$.
\end{definition}

\noindent In other words, we have 
\[ \Brap_{}^{\ts 2}(\RR^d) \, = \, \overline{\, \Big\{ \textstyle \sum_{k=1}^N c^{}_{k} \ee^{2\pi\ii a^{}_{k} \|x\|} \, : \, c^{}_{k} \in \CC, \ a^{}_{k}>0, \ N\in \NN \Big\}\, }^{\|\cdot\|_{b,2}}.
\]

\begin{corollary}
    The space $\Brap_{}^{\ts 2}(\RR^d)$ is a closed subspace of $BL^{2}_{}(\RR^d)$.  
\end{corollary}
\begin{proof}
    This immediately follows from the definition of the space.
\end{proof}

\begin{remark}
\label{rem:MAP}
One could try to define a radial version of almost periodicity by imposing the usual almost periodic conditions on the one-dimensional projection (the even extension of the radial part to the negative direction) of the studied function. Unfortunately, this approach does not seem to be suitable for our purposes, as the following example demonstrates. The simple spherical wave 
\[ f^{}_{a}(x) \, = \, \ee^{2\pi \ii a \|x \|}\ts ,  \]
belongs to the space of Besicovitch radially almost periodic functions. 
Its radial part, $\ee^{2\pi\ii a r}$, extended to the entire real line gives the even function $\ee^{2\pi\ii a |x|}$.  It is natural to ask if this function belongs to $\Bap^{\, 2}_{\cA_1}(\RR)$ (along $\cA_1 = \bigl([-R,R]\bigr)^{}_{R>0}$). A~simple computation shows that this is \emph{not} the case. 

However, it is worth mentioning that the function $\ee^{2\pi\ii a |x|}$ is a \emph{mean almost periodic function}, which is easy to show. For more details on mean almost periodicity, see \cite[Sec.~3]{LSS2020}. \exend
\end{remark}

If we consider the same equivalence relation as for the usual Besicovitch almost periodic functions, i.e., $ f \equiv g \,  \Longleftrightarrow \, \|f - g\|^{}_{b,2} = 0 \,$, then the space $\Brap^{\,2}(\RR^d) / \equiv$ is again a Hilbert space and we obtain an analogous statement to Theorem \ref{thm:hilbert}.  We mention that, although the equivalence class of a radial function in $\Brap_{}^{\ts 2}(\RR^d)$ may contain functions that are not radial, one can always work with a radial representative.

\begin{proposition}
\label{prop:Hilbert_spherical}
With the above notation, the space $\Brap_{}^{\ts 2}(\RR^d) / \equiv$ equipped with an inner product defined by
\[
    \langle [f], [g] \rangle^{}_{} \, \defeq \, M^{}_{}(f\ts \overline{g}\ts) \, = \, \lim_{R \rightarrow \infty} \myfrac{1}{\vol(B^{}_{R})} \int_{B^{}_{R}} f(t) \overline{g(t)} \,\dd t \ts ,
\]
is a Hilbert space. Moreover, the set 
\[ 
\bigl\{ [\ts \ee^{2\pi \ii a \|x \|} \ts] \, : \, a\in \RR \bigr\} \ts ,
\]
forms its orthonormal basis.
\end{proposition}
\begin{proof}
We first show that the mean $M(f\ts \overline{g}\ts)$ is well defined for all $f,g\in \Brap_{}^{\ts 2}(\RR^d)$.
By the Cauchy--Schwarz inequality, we obtain 
\[
\myfrac{1}{\vol(B^{}_{R})} \int_{B^{}_{R}} f(t) \overline{g(t)} \,\dd t \, \leqslant \, \myfrac{1}{\vol(B^{}_{R})} \biggl(\int_{B^{}_{R}} |f(t)|^2  \,\dd t \biggr)^{\frac{1}{2}} \biggl(\int_{B^{}_{R}} |g(t)|^2  \,\dd t \biggr)^{\frac{1}{2}} \, .
\]
Taking the limit $R \to \infty$ yields 
\[M(f\ts \overline{g}\ts) \, \leqslant \, \|f\|^{}_{b,2} \ts \|g\|^{}_{b,2} \, < \, \infty \, , \]
 as $\Brap_{}^{\ts 2}(\RR^d) \subseteq~BL^2(\RR^d)$, which shows that the mean exists. Independence of the choice of representative is clear. 

The mapping $\langle \ts \cdot \ts , \ts \cdot \ts \rangle$ satisfies the properties of an inner product, as one can directly check. The space $\Brap_{}^{\ts 2}(\RR^d) / \equiv$ is complete with respect to the norm induced by the inner product by construction, which is simply $\|\cdot\|^{}_{b,2}$. 

A direct computation yields 
\[ 
\bigl\langle [\ts \ee^{2\pi \ii a \|x \|} \ts],\ts[\ts \ee^{2\pi \ii b \|x \|} \ts] \bigr\rangle \, = \, \begin{cases}
    1, & \mbox{ if } a=b,\\
    0, & \mbox{ otherwise.}
\end{cases}
\]
Here, one benefits from the estimate \eqref{eq:R_n-1}. As the set $\bigl\{ [\ts \ee^{2\pi \ii a \|x \|} \ts] \, : \, a\in \RR \bigr\}$ is, by construction, a total set in $\Brap_{}^{\ts 2}(\RR^d) / \equiv \,$, the claim follows. 
\end{proof}

Since $\Brap_{}^{\ts 2}(\RR^d) / \equiv$ is a Hilbert space, we get the following as a consequence of the Riesz--Fischer property and the Parseval identity. 

\begin{corollary}
\label{thm:characterisation_Brap}
Let $F \in BL^2_{}(\RR^d)$.  Then $F \in \Brap^2(\RR^d)$ if and only if there exist sequences $c^{}_n \in \CC$ and $a^{}_n \in \RR$ such that
\[
    F(x) \, = \, \sum_{k=1}^{\infty} c^{}_{k} \ee^{2\pi\ii a^{}_{k} \|x\|} \ts ,
\]
in $(\Brap^2(\RR^d), \|\cdot\|^{}_{b,2})$.  Moreover, in this case, we have
\[
    c^{}_{k} \, = \, \bigl\langle F, f^{}_{a^{}_{k}} \bigr\rangle \,,
\]
where $f^{}_{a^{}_{k}}(x) = \ee^{2\pi\ii a^{}_{k} \|x\|}$, and 
\pushQED{\qed}
\[
    \|F\|^{}_{b,2} \, = \, \sum_{k=1}^\infty |c^{}_{k}|^2 \ts . \qedhere
\]
\popQED
\end{corollary}

Despite Remark~\ref{rem:MAP}, there is an alternate description of  $\Brap_{}^{\ts 2}(\RR^d)$ using the radial part of the function. For that, one has to consider the Besicovitch space $\Bap^{\, 2}_{\mathcal{B}}(\RR)$ with $\mathcal{B} = \bigl([0,R] \bigr)^{}_{R>0}$. Note that Proposition~\ref{prop:radial_ineq} remains true if we switch to the van Hove sequence $\mathcal{B}$. 

\begin{fact}
The space $\Bap^{\, 2}_{\mathcal{B}}(\RR)$ is (after taking the quotient by $\equiv$) a Hilbert space. The inner product on this space will be denoted by $\llangle \cdot \ts, \ts \cdot \rrangle^{}_{\mathcal{B}}$. With respect to this inner product, 
\[ 
\bigl\{ [\ts \ee^{2\pi \ii a r} \ts] \, : \, a\in \RR \bigr\} 
\]
is an orthonormal basis. \exend
\end{fact}

Now, the map $\Psi$ assigning to each radial profile function a rotationally invariant function in~$d$ dimensions provides an isometric isomorphism of spaces $\Bap^{\, 2}_{\mathcal{B}}(\RR)$ and $\Brap_{}^{\ts 2}(\RR^d)$. Indeed, it maps one orthonormal basis to another bijectively while preserving the orthonormality of the basis elements, as shown by a direct computation. 

\begin{theorem}
For any dimension $d\geq 1$, the map 
\begin{align*}
        \Psi: &\, \bigl(\Bap^{\, 2}_{\mathcal{B}}(\RR) / \equiv\ts , \, \llangle \cdot , \ts \cdot \rrangle^{}_{\mathcal{B}} \bigr) \, \longrightarrow \, \bigl(\Brap_{}^{\ts 2}(\RR^d) / \equiv\ts , \, \langle \cdot , \ts \cdot \rangle\bigr) \ts ,
\end{align*}
defined by 
\[\Psi(f)(x) \, \defeq \, f(\|x\|) \ts ,\] 
constitutes an isometric isomorphism of Hilbert spaces. \qed
\end{theorem}
As a slightly surprising by-product, we have obtained the following identity for averages of $\Bap^{\, 2}_{\mathcal{B}}(\RR)$ functions. 

\begin{corollary}
    For all $f,\, g \in \Bap^{\, 2}_{\mathcal{B}}(\RR)$ and for any dimension $d\geqslant 1$, we have 
    \[ 
    \lim_{R\to \infty} \myfrac{1}{R}\int^{R}_0 f(r) \overline{g(r)} \ts \dd r \, = \, \lim_{R\to \infty} \myfrac{d}{R^d}\int^{R}_0 f(r) \overline{g(r)} \ts r^{d-1} \ts \dd r \ts.
    \]
\end{corollary}
\begin{proof}
    For every $n$ and for all $f,\, g \in \Bap^{\, 2}_{\mathcal{B}}(\RR)$, the isometric isomorphism $\Psi$ gives
    \[\llangle f \ts, \ts  g \rrangle^{}_{\mathcal{B}} \, = \, \bigl\langle \Psi(f) \ts , \ts \Psi(g) \bigr\rangle, \]
    and the integration in spherical coordinates gives the prefactor.
\end{proof}
Now, we can profit from this alternative characterisation, as it provides the key argument for the existence of an autocorrelation. We will mimic the strategy we used for the planar waves by first arguing that the autocorrelation of the approximants exists and converges uniformly to a function, which we then declare to be the autocorrelation. 

\begin{proposition}
\label{prop:spherical_auto_ex}
    For every $F\in \Brap_{}^{\ts 2}(\RR^d)$, the autocorrelation function $\eta^{}_{F}$ exists and $\eta^{}_{F} \in \Cu(\RR^d)$. 
\end{proposition}
\begin{proof}
    Since $F\in \Brap^{2}(\RR^d)$, there exists a sequence of approximating spherical waves $S^{}_{N}$ in $\Cu(\RR^d)$ of the form $S^{}_{N}(x) = \sum_{k=1}^{N} c^{}_{k} \ee^{2\pi\ii a^{}_{k} \|x\|}$, such that $\|F-S^{}_{N}\|^{}_{b,2} \to 0$ as $N\to \infty$. Clearly, this sequence of approximants is Cauchy with respect to $\|\cdot\|_{b,2}$. 

    Now, by Theorem~\ref{thm:sphere-d}, 
    \[\lb S^{}_{N},S^{}_{N}\rb (x)\,  =\, \vG \! \left( \myfrac{d}{2}\right) \sum_{k=1}^N |c^{}_{k}|^2
    \frac{J^{}_{\frac{d}{2}-1}\bigl( 2 \pi \| a^{}_{k}x \|\bigr)}
    {\bigl( \pi \| a^{}_{k}x \| \bigr)^{\frac{d}{2} - 1}} \ts. \]
    Since for any $z \in \CC$ and for any $\nu \geqslant -\frac{1}{2}$, the Bessel function $J^{}_{\nu}$ satisfies \cite[Eq.~9.1.62]{AS1972}
    \[|J_{\nu}(z)| \, \leqslant \, \myfrac{\bigl|\tfrac{1}{2}z\bigr|^{\nu} \ts \ee^{|\mathrm{Im}(z)|}}{\vG(\nu+1)} \ts ,\]
    we obtain 
    \[\bigl| \lb S^{}_{N},S^{}_{N}\rb (x) \bigr| \, \leqslant \,  \vG \! \left( \myfrac{d}{2}\right) \sum_{k=1}^N |c^{}_{k}|^2
    \frac{\frac{1}{\vG\left(\frac{d}{2}\right)}\Bigl(\tfrac{2 \pi \| a^{}_{k}x \|}{2}\Bigr)^{\frac{d}{2}-1}}
    {\bigl( \pi \| a^{}_{k}x \| \bigr)^{\frac{d}{2} - 1}} \, = \, \sum_{k=1}^N |c^{}_{k}|^2 \, < \, \infty \ts , \]
    as follows from Theorem~\ref{thm:characterisation_Brap}. Therefore, $\lb S^{}_{N},S^{}_{N}\rb \in L^{\infty}(\RR^d) \cap BL^{2}(\RR^d)$ and we can apply the Cauchy--Schwartz inequality \eqref{eq:CS} and the usual trick (as in the proof of Theorem~\ref{thm:Bap}) to show 
    \begin{equation}\label{eq:S_N-conv-Cauchy}
    \bigl\|  \lb S^{}_{N}, S^{}_{N} \rb - \lb S^{}_{M}, S^{}_{M} \rb \bigr\|^{}_\infty \, \leqslant \,    2\|F \|^{}_{b,2} \ts \bigl\| S^{}_{N}-S^{}_{M} \bigr\|^{}_{b,2} \ts. 
    \end{equation}
    Moreover, since $S_N \in\Cu(\RR^d)$, we have that $\lb S_N, S_N \rb \in \Cu(\RR^d)$ by Proposition~\ref{prop:autocorrelation-reflected}, so \eqref{eq:S_N-conv-Cauchy} implies that the sequence $\lb S^{}_{N}, S^{}_{N} \rb$ is Cauchy in $\Cu(\RR^d)$. Therefore, by completeness of $\Cu(\RR^d)$, a uniformly continuous limiting function for this sequence exists, which we denote by $\eta^{}_{F}$; this function is the autocorrelation. 
\end{proof}

We can now prove the following radial analogy of Theorem~\ref{thm:Besicovitch-diffraction}. 

\begin{theorem}\label{thm:Brap-diffraction}
Let $F\in \Brap^{2}(\RR^d)$ and suppose that $F$ can be written as
\[
    F(x) \, = \, \sum_{m=1}^\infty c^{}_m \ee^{2\pi \ii a^{}_m \|x\|} \, ,
\]
where the convergence of the sum is with respect to $\|\cdot \|^{}_{b,2}$. Then, $\eta^{}_F$ exists and the diffraction reads
\[
    \widehat{\eta^{}_{F}} \, = \, \sum_{m=1}^\infty |c^{}_m|^2 \theta_{|a^{}_m|} \,.
\]
In particular, the diffraction is radially concentrated and rotationally invariant.
\end{theorem}

\begin{proof}
The autocorrelation exists and is translation bounded by Proposition \ref{prop:spherical_auto_ex}. For every approximant $S^{}_{N}$, the diffraction exists and reads $\widehat{\eta^{}_{S^{}_{N}}}  = \sum_{m=1}^N |c^{}_m|^2 \theta_{|a^{}_m|}$. 

We will now profit from the linearity of the Fourier transform as an isometry of $\cS'(\RR^d)$, the space of tempered distributions. 
To do so, observe that $\eta^{}_{S^{}_{N}}\in\cS'(\RR^d)$ for all $N$ as one can show easily. 
Moreover, we have that $\eta^{}_F = \lim_N\gamma^{}_{S^{}_N}$ is in $\Cu(\RR^d)$. As $\Cu(\RR^d) \subseteq \Lloc{1}(\RR^d) \hookrightarrow \cS'(\RR^d)$, we also obtain convergence in the sense of tempered distributions, or, explicitly, for any $\phi\in\cS(\RR^d)$, we have 
\[ \bigl|\langle \eta^{}_F -\eta^{}_{S^{}_{N}} , \phi\rangle\bigr| \, =  \, \left| \int^{}_{\RR^d} \bigl(\eta^{}_{F}(x)-\eta^{}_{S^{}_N}(x)\bigr)\ts\phi(x)\dd x\right| \, \leqslant  \, \bigl\|\eta^{}_F -\eta^{}_{S^{}_{N}} \bigr\|^{}_{\infty} \, \bigl\|\phi \bigr\|^{}_{1} \xrightarrow{N\to \infty} 0 \ts . \]

Therefore, the continuity of Fourier transform on $\cS'(\RR^d)$ implies $\widehat{\eta^{}_{S^{}_{n}}} \longrightarrow \widehat{\eta^{}_{F}}$, so 
\[ \widehat{\eta^{}_{F}} \, = \, \lim_{N\to \infty} \sum_{m=1}^N |c^{}_{m}|^2 \, \theta_{|a^{}_m|} \, =\, \sum_{m=1}^{\infty} |c^{}_{m}|^2 \, \theta_{|a^{}_m|} \ts .\]
The last equality can be established along the same line as in the proof of Theorem~\ref{thm:Besicovitch-diffraction}, while using the rotational symmetry of the function $F$ and its autocorrelation. 
\end{proof}

\begin{remark}
As mentioned earlier in Remark~\ref{rem:spherical_uniq_problem}, given a diffraction pattern consisting of concentric spheres with prescribed intensities, we cannot find (up to phase factors) a unique spherical wave with such a diffraction. However, if we impose the additional restriction that the wave be real, we gain uniqueness once we choose the phase factors. Indeed, as $\ee^{2\pi\ii r\|x\|} + \ee^{-2\pi\ii r\|x\|} = 2\cos(2\pi r \|x\|)$, the \emph{unique} (up to phase factors $(\xi^{}_{k})^{}_{k\in \NN}$) real spherical wave with diffraction 
\[
\sum^{\infty}_{k=1} C^{}_k \theta^{}_{a^{}_{k}} \ts , 
\]
reads
\[ 
\sum_{k=1}^{\infty} \xi^{}_{k} \sqrt{2C^{}_{k}} \ts \cos\bigl(2\pi a^{}_k\|x\|\bigr) \ts,
\]
where $\xi^{}_{k} \in \{\pm 1\}$, $C^{}_k >0$ and $a^{}_{k}>0$ for all $k\in \NN$. \exend
\end{remark}

So far, we have found functions whose diffraction is a point or a~sphere of radius $r$ centred at the origin. With the following simple trick, we can also obtain a function whose diffraction consists of a single circle of radius $r>0$ centred at a given point $a\in\RR^d$.

\begin{proposition}\label{prop:shift}
    Let $r \neq 0$ and $a \in \RR^d$. Then the diffraction of the measure $f^{}_{r,a}\lambda$ with density function $f^{}_{r,a}(x) = \ee^{2\pi \ii a\cdot x} \ee^{2\pi \ii r\| x\|}$ is the uniform probability distribution on the sphere of radius $r$ centred at $a$. Moreover, the following statements hold:
    \begin{enumerate}[label=(\alph*)]
    \item If $r'$ and $a'$ are such that $(r,a) \neq ( \pm r', a')$, then $\lb f_{r,a}, f_{f', a'}\rb = 0$;
    \item For any plane wave $g_b(x) = \ee^{2\pi \ii b \cdot x}$, we have $\lb f_{r,a}, g_b \rb = 0 $.
    \end{enumerate}
\end{proposition}
\begin{proof}
    With the notation used in Theorem \ref{thm:sphere-d}, a direct computation yields
    \[ \lb f^{}_{r,a}\ts , \ts f^{}_{r,a}\rb\, (x)  \, = \, \ee^{2\pi \ii a\cdot x}\ts  \eta^{}_{f^{}_{r}}(x)\ts .\]
    Taking the Fourier transform and considering its behaviour when multiplying by a~character gives the first claim.  

    Next, to prove (a) and (b), it suffices to verify that the diffraction measures are mutually singular by Theorem~\ref{thm:orthogonal}.  In the case of (a), observe that when $(r,a) \neq (r',a')$, the diffraction measures of $f_{r,a}\lambda$ and $f_{r',a'}\lambda$ are spheres with at most two points of intersection.  For (b), observe that the diffraction measure of $f_{r,a}$ and the diffraction measure of $g_b$ consist of a sphere and a point, respectively.  In both situations, the diffraction measures are indeed mutually singular.
\end{proof}

\begin{example}
    Using linearity and Proposition~\ref{prop:shift}, one can construct functions with a prescribed diffraction pattern consisting of spheres and points. For example, in two dimensions, one can consider 
    \[
    f^{}_{\mathrm{surprised}}(x) \, = \, \ee^{2\pi\ii \|x\|} + \ee^{6\pi\ii \|x\|} + (1+\ee^{4\sqrt{2}\pi\ii x^{}_{1}})\ee^{-2\sqrt{2}\pi\ii (x^{}_{1} - x^{}_{2})}\, ,   
    \]
    which we call a function with surprised diffraction: \vspace{2ex}

    \begin{center}
    \includegraphics{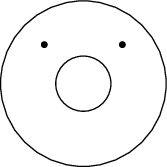}
    \end{center}

    One can also consider 
    \[
    f^{}_{\mathrm{Olympic}}(x) \, = \, \bigl(1 + 2\cos(4\pi x^{}_{1}) +  (1+\ee^{4\pi\ii x^{}_{1}})\ee^{-2\pi\ii (x^{}_{1} + x^{}_{2})}\bigr)\ee^{4\pi\ii\|x\|},
    \]
    whose diffraction consists of the Olympic rings: \vspace{2ex}

    \begin{center}
    \includegraphics{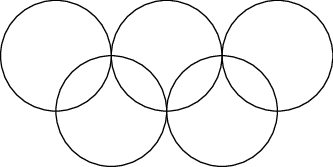}
    \end{center}
    \exend
\end{example}

\section*{Acknowledgment}
We thank Michael Baake and Nicolae Strungaru for their helpful discussions and comments on a draft of this paper. 
JM was supported by the German Research Council (Deutsche Forschungsgemeinschaft, DFG) under CRC 1283/2 (2021--317210226).

\printbibliography

\end{document}